\documentclass[11pt]{article}
\usepackage{amsmath,amssymb}
\usepackage{graphicx}
\usepackage{amsfonts}

\usepackage{amsthm}
\newtheorem{thm}{Theorem}[section]
\newtheorem{prop}[thm]{Proposition}

\usepackage[USenglish]{babel} 
\usepackage[T1]{fontenc}
\usepackage[ansinew]{inputenc}
\usepackage{mathrsfs}
\usepackage{stmaryrd} 
\usepackage{epstopdf} 
\usepackage{verbatim}
\usepackage{MnSymbol}  
\usepackage{accents}   
\usepackage{bbm}			 
\usepackage{color}     
\usepackage{authblk}   

\numberwithin{equation}{section}
\numberwithin{figure}{section}

\DeclareMathOperator*{\ext}{ext}

\title{Stochastic Discrete Hamiltonian Variational Integrators}
\date{}
\author[1]{Darryl D. Holm\thanks{\texttt{d.holm@imperial.ac.uk}}}
\author[1,2]{Tomasz M. Tyranowski\thanks{\texttt{tomasz.tyranowski@ipp.mpg.de}}}

\affil[1]{\small Mathematics Department\authorcr Imperial College London \authorcr London SW7 2AZ, UK}
\affil[2]{\small Max-Planck-Institut f\"ur Plasmaphysik \authorcr Boltzmannstra{\ss}e 2, 85748 Garching, Germany }

\begin{document}

\maketitle

\begin{abstract}
Variational integrators are derived for structure-preserving simulation of stochastic Hamiltonian systems with a certain type of multiplicative noise arising in geometric mechanics. The derivation is based on a stochastic discrete Hamiltonian which approximates a type-II stochastic generating function for the stochastic flow of the Hamiltonian system. The generating function is obtained by introducing an appropriate stochastic action functional and its corresponding variational principle. Our approach permits to recast in a unified framework a number of integrators previously studied in the literature, and presents a general methodology to derive new structure-preserving numerical schemes. The resulting integrators are symplectic; they preserve integrals of motion related to Lie group symmetries; and they include stochastic symplectic Runge-Kutta methods as a special case. Several new low-stage stochastic symplectic methods of mean-square order 1.0 derived using this approach are presented and tested numerically to demonstrate their superior long-time numerical stability and energy behavior compared to nonsymplectic methods.
\end{abstract}

\section{Introduction}
\label{sec:intro}
Stochastic differential equations (SDEs) play an important role in modeling dynamical systems subject to internal or external random fluctuations. Standard references include \cite{ArnoldSDE}, \cite{IkedaWatanabe1989}, \cite{KloedenPlatenSDE}, \cite{Kunita1997}, \cite{MilsteinBook}, \cite{ProtterStochastic}. Within this class of problems, we are interested in stochastic Hamiltonian systems, which take the form (see \cite{Bismut}, \cite{LaCa-Or2008}, \cite{MilsteinRepin2001})

\begin{align}
\label{eq: Stochastic Hamiltonian system}
dq &= \phantom{-}\frac{\partial H}{\partial p}dt + \frac{\partial h}{\partial p}\circ dW(t), \nonumber \\
dp &= -\frac{\partial H}{\partial q}dt - \frac{\partial h}{\partial q}\circ dW(t),
\end{align}

\noindent
where $H=H(q,p)$ and $h=h(q,p)$ are the Hamiltonian functions, $W(t)$ is the standard one-dimensional Wiener process, and $\circ$ denotes Stratonovich integration. The system \eqref{eq: Stochastic Hamiltonian system} can be formally regarded as a classical Hamiltonian system with the randomized Hamiltonian given by $\widehat H(q,p) = H(q,p) + h(q,p)\circ \dot W$, where $H(q,p)$ is the deterministic Hamiltonian and  $h(q,p)$ is \textit{another} Hamiltonian, to be specified, which multiplies (in the Stratonovich sense, denoted as $\circ$) a one-dimensional Gaussian white noise, $\dot W$. Such systems can be used to model, e.g., mechanical systems with uncertainty, or error, assumed to arise from random forcing, limited precision of experimental measurements, or unresolved physical processes on which the Hamiltonian of the deterministic system might otherwise depend. Particular examples include modeling synchrotron oscillations of particles in particle storage rings (see \cite{SeesselbergParticleStorageRings}, \cite{DomeAccelerators}) and stochastic dynamics of the interactions of singular solutions of the EPDiff basic fluids equation (see \cite{HolmTyranowskiSolitons}). More examples are discussed in Section~\ref{sec:Numerical experiments}. See also \cite{LelievreStoltz2016}, \cite{Mao2007}, \cite{Nelson1988}, \cite{SanzSerna1999}, \cite{Shardlow2003}, \cite{Soize1994}, \cite{Talay2002}.

As occurs for other SDEs, most Hamiltonian SDEs cannot be solved analytically and one must resort to numerical simulations to obtain approximate solutions. In principle, general purpose stochastic numerical schemes for SDEs can be applied to stochastic Hamiltonian systems. However, as for their deterministic counterparts, stochastic Hamiltonian systems possess several important geometric features. In particular, their phase space flows (almost surely) preserve the symplectic structure. When simulating these systems numerically, it is therefore advisable that the numerical scheme also preserves such geometric features. Geometric integration of deterministic Hamiltonian systems has been thoroughly studied (see \cite{HLWGeometric}, \cite{McLachlanQuispel}, \cite{SanzSerna} and the references therein) and symplectic integrators have been shown to demonstrate superior performance in long-time simulations of Hamiltonian systems, compared to non-symplectic methods; so it is natural to pursue a similar approach for stochastic Hamiltonian systems. This is a relatively recent pursuit. Stochastic symplectic integrators were first proposed in \cite{MilsteinRepin2001} and \cite{MilsteinRepin}. Stochastic generalizations of symplectic partitioned Runge-Kutta methods were analyzed in \cite{Burrage2012}, \cite{MaDing2012}, and \cite{MaDing2015}. A stochastic generating function approach to constructing stochastic symplectic methods, based on approximately solving a corresponding stochastic Hamilton-Jacobi equation satisfied by the generating function, was proposed in \cite{WangPHD} and \cite{Wang2014}, and this idea was further pursued in \cite{AntonWeak2014}, \cite{Anton2013}, \cite{AntonHighOrder2014}. Stochastic symplectic integrators constructed via composition methods were proposed and analyzed in \cite{Misawa2010}. A first order weak symplectic numerical scheme and an extrapolation method were proposed and their global error was analyzed in~\cite{AntonGlobalError2013}. More recently, an approach based on Pad\'e approximations has been used to construct stochastic symplectic methods for linear stochastic Hamiltonian systems (see \cite{SunWang2016}). Higher-order strong and weak symplectic partitioned Runge-Kutta methods have been proposed in \cite{Wang2017} and \cite{Zhou2017}. High-order conformal symplectic and ergodic schemes for the stochastic Langevin equation have been introduced in \cite{Hong2017}. Other structure-preserving methods for stochastic Hamiltonian systems have also been investigated, see, e.g., \cite{Anmarkrud2017}, \cite{Burrage2014}, \cite{Hong2015}, and the references therein.

Long-time accuracy and near preservation of the Hamiltonian by symplectic integrators applied to deterministic Hamiltonian systems have been rigorously studied using the so-called backward error analysis (see, e.g., \cite{HLWGeometric} and the references therein). To the best of our knowledge, such rigorous analysis has not been attempted in the stochastic context as yet. However, the numerical evidence presented in the papers cited above is promising and suggests that stochastic symplectic integrators indeed possess the property of very accurately capturing the evolution of the Hamiltonian $H$ over exponentially long time intervals (note that the Hamiltonian $H$ in general does not stay constant for stochastic Hamiltonian systems).

An important class of geometric integrators are \emph{variational integrators}. This type of numerical schemes is based on discrete variational principles and provides a natural framework for the discretization of Lagrangian systems, including forced, dissipative, or constrained ones. These methods have the advantage that they are symplectic, and in the presence of a symmetry, satisfy a discrete version of Noether's theorem. For an overview of variational integration for deterministic systems see \cite{MarsdenWestVarInt}; see also \cite{HallLeokSpectral}, \cite{LeokShingel}, \cite{LeokZhang}, \cite{OberBlobaum2016}, \cite{OberBlobaum2015}, \cite{RowleyMarsden}, \cite{TyranowskiDesbrunLinearLagrangians}, \cite{VankerschaverLeok}. Variational integrators were introduced in the context of finite-dimensional mechanical systems, but were later generalized to Lagrangian field theories (see \cite{MarsdenPatrickShkoller}) and applied in many computations, for example in elasticity, electrodynamics, or fluid dynamics; see \cite{LewAVI}, \cite{Pavlov}, \cite{SternDesbrun}, \cite{TyranowskiDesbrunRAMVI}. 

Stochastic variational integrators were first introduced in \cite{BouRabeeSVI} and further studied in \cite{BouRabeeConstrainedSVI}. However, those integrators were restricted to the special case when the Hamiltonian function $h=h(q)$ was independent of $p$, and only low-order Runge-Kutta types of discretization were considered. In the present work we extend the idea of stochastic variational integration to general stochastic Hamiltonian systems by generalizing the variational principle introduced in \cite{LeokZhang} and applying a Galerkin type of discretization (see \cite{MarsdenWestVarInt}, \cite{LeokShingel}, \cite{LeokZhang}, \cite{OberBlobaum2015}, \cite{OberBlobaum2016}), which leads to a more general class of stochastic symplectic integrators than those presented in \cite{BouRabeeConstrainedSVI}, \cite{BouRabeeSVI}, \cite{MaDing2012}, and \cite{MaDing2015}. Our approach consists in approximating a generating function for the stochastic flow of the Hamiltonian system, but unlike in \cite{WangPHD} and \cite{Wang2014}, we do make this discrete approximation by exploiting its variational characterization, rather than solving the corresponding Hamilton-Jacobi equation.

\paragraph{Main content}
The main content of the remainder of this paper is, as follows. 
\begin{description}
\item
In Section~\ref{sec:Variational principle for stochastic Hamiltonian systems} we introduce a stochastic variational principle and a stochastic generating function suitable for considering stochastic Hamiltonian systems, and we discuss their properties. 
\item
In Section~\ref{sec:Stochastic Galerkin Hamiltonian Variational Integrators} we present a general framework for constructing stochastic Galerkin variational integrators, prove the symplecticity and conservation properties of such integrators, show they contain the stochastic symplectic Runge-Kutta methods of \cite{MaDing2012}, \cite{MaDing2015} as a special case, and finally present several particularly interesting examples of new low-stage stochastic symplectic integrators of mean-square order $1.0$ derived with our general methodology. 
\item
In Section~\ref{sec:Numerical experiments} we present the results of our numerical tests, which verify the theoretical convergence rates and the excellent long-time performance of our integrators compared to some popular non-symplectic methods. 
\item
Section~\ref{sec:Summary} contains the summary of our work.
\end{description}

\section{Variational principle for stochastic Hamiltonian systems}
\label{sec:Variational principle for stochastic Hamiltonian systems}

The stochastic variational integrators proposed in \cite{BouRabeeSVI} and \cite{BouRabeeConstrainedSVI} were formulated for dynamical systems which are described by a Lagrangian and which are subject to noise whose magnitude depends only on the position $q$. Therefore, these integrators are applicable to \eqref{eq: Stochastic Hamiltonian system} only when the Hamiltonian function $h=h(q)$ is independent of $p$ and the Hamiltonian $H$ is non-degenerate (i.e., the associated Legendre transform is invertible). However, in the case of general $h=h(q,p)$ the paths $q(t)$ of the system become almost surely nowhere differentiable, which poses a difficulty in interpreting the meaning of the corresponding Lagrangian. Therefore, we need a different sort of action functional and variational principle to construct stochastic symplectic integrators for \eqref{eq: Stochastic Hamiltonian system}. To this end, we will generalize the approach taken in \cite{LeokZhang}. To begin, in the next section, we will introduce an appropriate stochastic action functional and show that it can be used to define a type-II generating function for the stochastic flow of the system \eqref{eq: Stochastic Hamiltonian system}.

\subsection{Stochastic variational principle}
\label{sec:Stochastic variational principle}

Let the Hamiltonian functions $H: T^*Q \longrightarrow \mathbb{R}$ and $h: T^*Q \longrightarrow \mathbb{R}$ be defined on the cotangent bundle $T^*Q$ of the configuration manifold $Q$, and let $(q,p)$ denote the canonical coordinates on $T^*Q$. For simplicity, in this work we assume that the configuration manifold has a vector space structure, $Q \cong \mathbb{R}^N$, so that $T^*Q = Q \times Q^* \cong \mathbb{R}^N \times \mathbb{R}^N$ and $TQ = Q \times Q \cong \mathbb{R}^N \times \mathbb{R}^N$. In this case, the natural pairing between one-forms and vectors can be identified with the scalar product on $\mathbb{R}^N$, that is, $\langle (q,p), (q,\dot q) \rangle = p\cdot \dot q$, where $(q,\dot q)$ denotes the coordinates on $TQ$ . Let $(\Omega, \mathcal{F},\mathbb{P})$ be the probability space with the filtration $\{\mathcal{F}_t\}_{t \geq 0}$, and let $W(t)$ denote a standard one-dimensional Wiener process on that probability space (such that $W(t)$ is $\mathcal{F}_t$-measurable). We will assume that the Hamiltonian functions $H$ and $h$ are sufficiently smooth and satisfy all the necessary conditions for the existence and uniqueness of solutions to \eqref{eq: Stochastic Hamiltonian system}, and their extendability to a given time interval $[t_a, t_b]$ with $t_b > t_a \geq 0$. One possible set of such assumptions can be formulated by considering the It\^o form of~\eqref{eq: Stochastic Hamiltonian system},

\begin{equation}
\label{eq: Ito form of the stochastic Hamiltonian system}
dz = A(z) dt + B(z)dW(t),
\end{equation}

\noindent
with $z=(q,p)$ and

\begin{align}
\label{eq: Ito coefficients}
A(z) =
\begin{pmatrix}
 \phantom{-}\frac{\partial H}{\partial p}+\frac{1}{2}\frac{\partial^2 h}{\partial p \partial q} \frac{\partial h}{\partial p}-\frac{1}{2}\frac{\partial^2 h}{\partial p^2} \frac{\partial h}{\partial q}  \\
-\frac{\partial H}{\partial q}-\frac{1}{2}\frac{\partial^2 h}{\partial q^2} \frac{\partial h}{\partial p}+\frac{1}{2}\frac{\partial^2 h}{\partial q \partial p} \frac{\partial h}{\partial q}
\end{pmatrix},
\qquad\qquad B(z) =
\begin{pmatrix}
 \phantom{-}\frac{\partial h}{\partial p} \\
-\frac{\partial h}{\partial q}
\end{pmatrix},
\end{align}

\noindent
where $\partial^2 h/\partial q^2$, $\partial^2 h/\partial p^2$, and $\partial^2 h/\partial q \partial p$ denote the Hessian matrices of $h$. For simplicity and clarity of the exposition, throughout this paper we assume that (see \cite{ArnoldSDE}, \cite{IkedaWatanabe1989}, \cite{KloedenPlatenSDE}, \cite{Kunita1997})

\begin{itemize}
	\item[(H1)] $H$ and $h$ are $C^2$ functions of their arguments
	\item[(H2)] $A$ and $B$ are globally Lipschitz
\end{itemize}

\noindent
These assumptions are sufficient\footnote{In this work we only consider Hamiltonian functions $H$ and $h$ that are independent of time. In the time-dependent case one needs to add a further assumption that the growth of $A$ and $B$ is linearly bounded, i.e. $\|A(z,t)\|^2+\|B(z,t)\|^2 \leq K (1 + \|z\|^2)$ for a constant $K>0$ (see \cite{ArnoldSDE}, \cite{IkedaWatanabe1989}, \cite{KloedenPlatenSDE}, \cite{Kunita1997}).} for our purposes, but could be relaxed if necessary. Define the space

\begin{equation}
\label{eq:Definition of the q space}
C([t_a, t_b]) = \big\{ (q,p):\Omega \times [t_a, t_b] \longrightarrow T^*Q  \, \big| \, \text{$q$, $p$ are almost surely continuous $\mathcal{F}_t$-adapted semimartingales} \big\}.
\end{equation}

\noindent
Since we assume $T^*Q \cong \mathbb{R}^N \times \mathbb{R}^N$, the space $C([t_a, t_b])$ is a vector space (see \cite{ProtterStochastic}). Therefore, we can identify the tangent space $TC([t_a, t_b]) \cong C([t_a, t_b])\times C([t_a, t_b])$. We can now define the following stochastic action functional, $\mathcal{B}: \Omega \times C([t_a, t_b]) \longrightarrow \mathbb{R}$,

\begin{equation}
\label{eq:Stochastic action functional}
\mathcal{B}\big[q(\cdot),p(\cdot) \big] = p(t_b)q(t_b) - \int_{t_a}^{t_b} \Big[ p\circ dq - H\big(q(t),p(t)\big)\,dt - h\big(q(t),p(t)\big)\circ dW(t)\Big],
\end{equation}

\noindent
where $\circ$ denotes Stratonovich integration, and we have omitted writing the elementary events $\omega \in \Omega$ as arguments of functions, following the standard convention in stochastic analysis. 

\begin{thm}[{\bf Stochastic Variational Principle in Phase Space}]
\label{thm:Stochastic Variational Principle}
Suppose that $H(q,p)$ and $h(q,p)$ satisfy conditions (H1)-(H2). If the curve $\big( q(t), p(t) \big)$ in $T^*Q$ satisfies the stochastic Hamiltonian system \eqref{eq: Stochastic Hamiltonian system} for $t\in [t_a,t_b]$, where $t_b \geq t_a >0$, then the pair $\big( q(\cdot), p(\cdot) \big)$ is a critical point of the stochastic action functional \eqref{eq:Stochastic action functional}, that is,

\begin{equation}
\label{eq:Stochastic Variational Principle}
\delta \mathcal{B}\big[q(\cdot), p(\cdot) \big] \equiv \frac{d}{d\epsilon} \bigg|_{\epsilon=0}\mathcal{B}\big[q(\cdot)+\epsilon \delta q(\cdot), p(\cdot)+\epsilon \delta p(\cdot) \big]=0\,,
\end{equation} 

\noindent
almost surely for all variations $\big(\delta q(\cdot), \delta p(\cdot) \big) \in C([t_a, t_b])$ such that almost surely $\delta q(t_a)=0$ and $\delta p(t_b)=0$.
\end{thm}

\begin{proof}
Let the curve $\big( q(t), p(t) \big)$ in $T^*Q$ satisfy \eqref{eq: Stochastic Hamiltonian system} for $t\in [t_a,t_b]$. It then follows that the stochastic processes $q(t)$ and $p(t)$ are almost surely continuous, $\mathcal{F}_t$-adapted semimartingales, that is, $\big( q(\cdot), p(\cdot) \big) \in C([t_a, t_b])$ (see \cite{ArnoldSDE}, \cite{ProtterStochastic}). We calculate the variation \eqref{eq:Stochastic Variational Principle} as

\begin{align}
\label{eq:Calculating delta B}
\delta \mathcal{B}\big[q(\cdot), p(\cdot) \big] &= p(t_b)\delta q(t_b) - \int_{t_a}^{t_b} p(t)\circ d \delta q(t) - \int_{t_a}^{t_b} \delta p(t)\circ d q(t) \nonumber \\
                                       &\phantom{=} + \int_{t_a}^{t_b} \bigg[ \frac{\partial H}{\partial q}\big(q(t),p(t)\big)\,\delta q(t) + \frac{\partial H}{\partial p}\big(q(t),p(t)\big)\,\delta p(t) \bigg]\,dt \nonumber \\
                                       &\phantom{=} + \int_{t_a}^{t_b} \bigg[ \frac{\partial h}{\partial q}\big(q(t),p(t)\big)\,\delta q(t) + \frac{\partial h}{\partial p}\big(q(t),p(t)\big)\,\delta p(t) \bigg]\circ dW(t),
\end{align}

\noindent
where we have used the end point condition, $\delta p(t_b)=0$. Since the Hamiltonians are $C^2$ and the processes $q(t)$, $p(t)$ are almost surely continuous, in the last two lines we have used a dominated convergence argument to interchange differentiation with respect to $\epsilon$ and integration with respect to $t$ and $W(t)$. Upon applying the integration by parts formula for semimartingales (see \cite{ProtterStochastic}), we find

\begin{equation}
\label{eq:Integration by parts for semimartingales}
\int_{t_a}^{t_b} p(t)\circ d \delta q(t) = p(t_b)\delta q(t_b) -  p(t_a)\delta q(t_a) - \int_{t_a}^{t_b} \delta q(t)\circ d p(t).
\end{equation}

\noindent
Substituting and rearranging terms produces,

\begin{align}
\label{eq:Calculating delta B continued}
\delta \mathcal{B}\big[q(\cdot), p(\cdot) \big] &= \int_{t_a}^{t_b} \delta q(t) \bigg[\circ dp(t) + \frac{\partial H}{\partial q}\big(q(t),p(t)\big)\,dt + \frac{\partial h}{\partial q}\big(q(t),p(t)\big)\circ dW(t) \bigg] \nonumber \\
&- \int_{t_a}^{t_b} \delta p(t) \bigg[\circ dq(t) - \frac{\partial H}{\partial p}\big(q(t),p(t)\big)\,dt - \frac{\partial h}{\partial p}\big(q(t),p(t)\big)\circ dW(t) \bigg],
\end{align}

\noindent
where we have used $\delta q(t_a)=0$. Since $\big( q(t), p(t) \big)$ satisfy \eqref{eq: Stochastic Hamiltonian system}, then by definition we have that almost surely for all $t \in [t_a,t_b]$,

\begin{equation}
\label{eq: Integral form of the solution of the stochastic Hamiltonian system}
q(t) = q(t_a) + \underbrace{\int_{t_a}^t \frac{\partial H}{\partial p}(q(s),p(s)) \,ds}_{M_1(t)} + \underbrace{\int_{t_a}^t \frac{\partial h}{\partial p}(q(s),p(s)) \circ dW(s)}_{M_2(t)},
\end{equation}

\noindent
that is, $q(t)$ can be represented as the sum of two semi-martingales $M_1(t)$ and $M_2(t)$, where the sample paths of the process $M_1(t)$ are almost surely continuously differentiable. Let us calculate

\begin{align}
\label{eq: One of the integrals in delta B}
\int_{t_a}^{t_b} \delta p(t)\circ dq(t) &= \int_{t_a}^{t_b} \delta p(t)\circ d\big(q(t_a)+M_1(t)+M_2(t)\big) \nonumber \\
                                        &= \int_{t_a}^{t_b} \delta p(t)\circ dM_1(t) + \int_{t_a}^{t_b} \delta p(t)\circ dM_2(t) \nonumber \\
																				&= \int_{t_a}^{t_b} \delta p(t)\frac{\partial H}{\partial p}(q(t),p(t)) \,dt + \int_{t_a}^{t_b}                                            \delta p(t) \frac{\partial h}{\partial p}(q(t),p(t)) \circ dW(t),
\end{align}

\noindent
where in the last equality we have used the standard property of the Riemann-Stieltjes integral for the first term, as $M_1(t)$ is almost surely differentiable, and the associativity property of the Stratonovich integral for the second term (see \cite{ProtterStochastic}, \cite{IkedaWatanabe1989}). Substituting \eqref{eq: One of the integrals in delta B} in \eqref{eq:Calculating delta B continued}, we show that the second term is equal to zero. By a similar argument we also prove that the first term in \eqref{eq:Calculating delta B continued} is zero. Therefore, $\delta \mathcal{B} = 0$, almost surely.\\
\end{proof}

\noindent
{\bf Remark:} It is natural to expect that the converse theorem, that is, if $\big( q(\cdot), p(\cdot) \big)$ is a critical point of the stochastic action functional \eqref{eq:Stochastic action functional}, then the curve $\big( q(t), p(t) \big)$ satisfies \eqref{eq: Stochastic Hamiltonian system}, should also hold, although a larger class of variations $(\delta q, \delta p)$ may be necessary. A variant of such a theorem, although for a slightly different variational principle and in a different setting, was proved in L\'azaro-Cam\'i \& Ortega \cite{LaCa-Or2008}. Another variant for Lagrangian systems was proved by Bou-Rabee \& Owhadi \cite{BouRabeeSVI} in the special case when $h=h(q)$ is independent of $p$. In that case, one can assume that $q(t)$ is continuously differentiable. In the general case, however, $q(t)$ is not differentiable, and the ideas of \cite{BouRabeeSVI} cannot be applied directly. We leave this as an open question. Here, we will use the action functional \eqref{eq:Stochastic action functional} and the variational principle \eqref{eq:Stochastic Variational Principle} to construct numerical schemes, and we will directly verify that these numerical schemes converge to solutions of \eqref{eq: Stochastic Hamiltonian system}.\\

\subsection{Stochastic type-II generating function}
\label{sec:Stochastic type-II generating function}

When the Hamiltonian functions $H(q,p)$ and $h(q,p)$ satisfy standard measurability and regularity conditions (e.g., (H1)-(H2)), then the system \eqref{eq: Stochastic Hamiltonian system} possesses a pathwise unique stochastic flow $F_{t,t_0}: \Omega \times T^*Q \longrightarrow T^*Q$. It can be proved that for fixed $t,t_0$ this flow is mean-square differentiable with respect to the $q$, $p$ arguments, and is also almost surely a diffeomorphism (see \cite{ArnoldSDE}, \cite{IkedaWatanabe1989}, \cite{KloedenPlatenSDE}, \cite{Kunita1997}). Moreover, $F_{t,t_0}$ almost surely preserves the canonical symplectic form $\Omega_{T^*Q} = \sum_{i=1}^N dq^i \wedge dp^i$ (see \cite{MilsteinRepin}, \cite{Bismut}, \cite{LaCa-Or2008}), that is,

\begin{equation}
\label{eq:Symplecticity of the Hamiltonian flow}
F^*_{t,t_0} \Omega_{T^*Q} = \Omega_{T^*Q},
\end{equation}

\noindent
where $F^*_{t,t_0}$ denotes the pull-back by the flow $F_{t,t_0}$. We will show below that the action functional \eqref{eq:Stochastic action functional} can be used to construct a type II generating function for $F_{t,t_0}$. Let $(\bar q(t), \bar p(t))$ be a particular solution of \eqref{eq: Stochastic Hamiltonian system} on $[t_a,t_b]$. Suppose that for almost all $\omega \in \Omega$ there is an open neighborhood $\mathcal{U}(\omega)\subset Q$ of $\bar q(\omega,t_a)$, an open neighborhood $\mathcal{V}(\omega)\subset Q^*$ of $\bar p(\omega,t_b)$, and an open neighborhood $\mathcal{W}(\omega)\subset T^*Q$ of the curve $(\bar q(\omega,t), \bar p(\omega,t))$ such that for all $q_a \in \mathcal{U}(\omega)$ and $p_b \in \mathcal{V}(\omega)$ there exists a pathwise unique solution $(\bar q(\omega,t; q_a, p_b), \bar p(\omega,t; q_a, p_b))$ of \eqref{eq: Stochastic Hamiltonian system} which satisfies $\bar q(\omega,t_a; q_a, p_b)=q_a$, $\bar p(\omega,t_b; q_a, p_b)=p_b$, and $(\bar q(\omega,t; q_a, p_b), \bar p(\omega,t; q_a, p_b)) \in \mathcal{W}(\omega)$ for $t_a\leq t \leq t_b$. (As in the deterministic case, for $t_b$ sufficiently close to $t_a$ one can argue that such neighborhoods exist; see \cite{MarsdenRatiuSymmetry}.) Define the function $S:\mathcal{Y} \longrightarrow \mathbb{R}$ as

\begin{equation}
\label{eq:Definition of the generating function}
S(q_a,p_b) = \mathcal{B}\big[\bar q(\cdot;q_a,p_b), \bar p(\cdot; q_a,p_b) \big],
\end{equation}

\noindent
where the domain $\mathcal{Y} \subset \Omega \times T^*Q$ is given by $\mathcal{Y} = \bigcup\limits_{\omega \in \Omega} \{\omega \} \times \mathcal{U}(\omega) \times \mathcal{V}(\omega)$. Below we prove that $S$ generates\footnote{A generating function for the symplectic transformation $(q_a,p_a)\longrightarrow (q_b,p_b)$ is a function of one of the variables $(q_a,p_a)$ and one of the variables $(q_b,p_b)$. Therefore, there are four basic types of generating functions: $S=S_1(q_a, q_b)$, $S=S_2(q_a, p_b)$, $S=S_3(p_a, q_b)$, and $S=S_4(p_a, p_b)$. In this work we use the type-II generating function $S=S_2(q_a, p_b)$.} the stochastic flow $F_{t_b,t_a}$. 

\begin{thm}
\label{thm:S generates the Hamiltonian flow}
The function $S(q_a,p_b)$ is a type-II stochastic generating function for the stochastic mapping $F_{t_b,t_a}$, that is, $F_{t_b,t_a}:(q_a,p_a)\longrightarrow (q_b,p_b)$ is implicitly given by the equations

\begin{equation}
\label{eq:Equations generating the flow of the Hamiltonian system}
q_b = D_2 S(q_a,p_b), \qquad\qquad p_a = D_1 S(q_a,p_b),
\end{equation}

\noindent
where the derivatives are understood in the mean-square sense.
\end{thm}

\begin{proof}
Under appropriate regularity assumptions on the Hamiltonians (e.g., (H1)-(H2)), the solutions $\bar q(t; q_a, p_b)$ and $\bar p(t; q_a, p_b)$ are mean-square differentiable with respect to the parameters $q_a$ and $p_b$, and the partial derivatives are semimartingales (see \cite{ArnoldSDE}). We calculate the derivative of $S$ as

\begin{align}
\label{eq:Calculate derivative of S}
\frac{\partial S}{\partial q_a}(q_a,p_b) &= p_b \frac{\partial \bar q(t_b)}{\partial q_a} - \int_{t_a}^{t_b} \frac{\partial \bar p(t)}{\partial q_a} \circ d \bar q(t)  - \int_{t_a}^{t_b} \bar p(t) \circ d \frac{\partial \bar q(t)}{\partial q_a} \nonumber \\
&\phantom{=} + \int_{t_a}^{t_b} \bigg[ \frac{\partial \bar q(t)}{\partial q_a} \frac{\partial H}{\partial q} \big(\bar q(t), \bar p(t) \big) + \frac{\partial \bar p(t)}{\partial q_a} \frac{\partial H}{\partial p} \big(\bar q(t), \bar p(t) \big) \bigg]\,dt \nonumber \\
&\phantom{=} + \int_{t_a}^{t_b} \bigg[ \frac{\partial \bar q(t)}{\partial q_a} \frac{\partial h}{\partial q} \big(\bar q(t), \bar p(t) \big) + \frac{\partial \bar p(t)}{\partial q_a} \frac{\partial h}{\partial p} \big(\bar q(t), \bar p(t) \big) \bigg]\circ dW(t),
\end{align}

\noindent
where for notational convenience we have omitted writing $q_a$ and $p_b$ explicitly as arguments of $\bar q(t)$ and $\bar p(t)$. Applying the integration by parts formula for semimartingales (see \cite{ProtterStochastic}), we find

\begin{equation}
\label{eq:Integration by parts for semimartingales 2}
\int_{t_a}^{t_b} \bar p(t) \circ d \frac{\partial \bar q(t)}{\partial q_a} = p_b\frac{\partial \bar q(t_b)}{\partial q_a} - \bar p(t_a) - \int_{t_a}^{t_b} \frac{\partial \bar q(t)}{\partial q_a} \circ d \bar p(t). 
\end{equation}

\noindent
Substituting and rearranging terms, we obtain the result,

\begin{align}
\label{eq:Calculate derivative of S continued}
\frac{\partial S}{\partial q_a}(q_a,p_b) = \bar p(t_a) &+ \int_{t_a}^{t_b}\frac{\partial \bar q(t)}{\partial q_a} \bigg[\circ d \bar p +  \frac{\partial H}{\partial q} \big(\bar q(t), \bar p(t) \big)\,dt + \frac{\partial h}{\partial q} \big(\bar q(t), \bar p(t) \big)\circ dW(t) \bigg] \nonumber \\
&+ \int_{t_a}^{t_b}\frac{\partial \bar p(t)}{\partial q_a} \bigg[\circ d \bar q -  \frac{\partial H}{\partial p} \big(\bar q(t), \bar p(t) \big)\,dt - \frac{\partial h}{\partial p} \big(\bar q(t), \bar p(t) \big)\circ dW(t) \bigg] 
\nonumber\\&= \bar p(t_a),
\end{align}

\noindent
since $(\bar q(t), \bar p(t))$ is a solution of \eqref{eq: Stochastic Hamiltonian system}. Similarly we show $\partial S / \partial p_b (q_a,p_b) = \bar q(t_b)$. By definition of the flow, then $F_{t_b,t_a}(q_a, \bar p(t_a)) = (\bar q(t_b), p_b)$.

\end{proof}

\noindent
We can consider $S(q_a,p_b)$ as a function of time if we let $t_b$ vary. Let us denote this function as $S_t(q_a,p)$. Below we show that $S_t(q_a,p)$ satisfies a certain stochastic partial differential equation, which is a stochastic generalization of the Hamilton-Jacobi equation considered in \cite{LeokZhang}.

\begin{prop}[{\bf Type II Stochastic Hamilton-Jacobi Equation}]
\label{thm:Type II Stochastic Hamilton-Jacobi equation}
Let the time-dependent type-II generating function be defined as

\begin{equation}
\label{eq:Type II time-dependent generating function}
S_2(q_a,p,t) \equiv S_t(q_a,p) = p \bar q(t) - \int_{t_a}^{t} \Big[ \bar p(\tau) \circ d\bar q(\tau) - H\big(\bar q(\tau),\bar p(\tau)\big)\,d\tau - h\big(\bar q(\tau),\bar p(\tau)\big)\circ dW(\tau)\Big],
\end{equation}

\noindent
where $\bar q(\tau) \equiv \bar q(\tau; q_a, p)$ and $\bar p(\tau) \equiv \bar p(\tau; q_a, p)$ as before. Then the function $S_2(q_a,p,t)$ satisfies the following stochastic partial differential equation

\begin{equation}
\label{eq:Type II Stochastic Hamilton-Jacobi equation}
d S_2 = H\Big(\frac{\partial S_2}{\partial p}, p \Big)\,dt + h\Big(\frac{\partial S_2}{\partial p}, p \Big)\circ dW(t),
\end{equation}

\noindent
where $dS_2$ denotes the stochastic differential of $S_2$ with respect to the $t$ variable.
\end{prop}

\begin{proof}
Choose an arbitrary pair $(q_a,p_a) \in T^*Q$ and define the particular solution $(\bar q(\tau), \bar p(\tau)) = F_{\tau,t_a}(q_a,p_a)$. Form the function $S_2(q_a, \bar p(t),t)$ and consider its total stochastic differential\footnote{In analogy to ordinary calculus, the total stochastic differential is understood as $S_2(q_a, \bar p(t_b),t_b)-S_2(q_a, \bar p(t_a),t_a) = \int_{t_a}^{t_b} \bar dS_2(q_a, \bar p(t),t)$, whereas the partial stochastic differential means $S_2(q_a, p_b,t_b)-S_2(q_a, p_b,t_a) = \int_{t_a}^{t_b} dS_2(q_a,p_b,t)$.} $\bar dS_2(q_a, \bar p(t),t)$ with respect to time. On one hand, the rules of Stratonovich calculus imply

\begin{equation}
\label{eq:dS2 from Stratonovich}
\bar dS_2(q_a, \bar p(t),t) = dS_2(q_a, \bar p(t),t) + \frac{\partial S_2}{\partial p}\Big( q_a, \bar p(t),t \Big)\circ d \bar p(t),
\end{equation} 

\noindent
where $dS_2$ denotes the partial stochastic differential of $S_2$ with respect to the $t$ argument. On the other hand, integration by parts in \eqref{eq:Type II time-dependent generating function} implies

\begin{equation}
\label{eq:dS2 from definition}
\bar dS_2(q_a, \bar p(t),t) = \bar q(t)\circ d \bar p(t) + H(\bar q(t),\bar p(t))\,dt + h(\bar q(t),\bar p(t))\circ dW(t).
\end{equation}

\noindent
Comparing \eqref{eq:dS2 from Stratonovich} and \eqref{eq:dS2 from definition}, and using \eqref{eq:Equations generating the flow of the Hamiltonian system}, we obtain

\begin{equation}
\label{eq:Stochastic Hamilton-Jacobi equation along a path}
d S_2(q_a, \bar p(t),t) = H\bigg(\frac{\partial S_2}{\partial p}\big(q_a, \bar p(t),t \big), \bar p(t) \bigg)\,dt + h\bigg(\frac{\partial S_2}{\partial p}\big(q_a, \bar p(t),t\big), \bar p(t) \bigg)\circ dW(t).
\end{equation}

\noindent
This equation is satisfied along a particular path $\bar p(t)$, however, as in the discussion preceding Theorem~\ref{thm:S generates the Hamiltonian flow}, we can argue, slightly informally, that for almost all $\omega \in \Omega$, and for $t$ sufficiently close to $t_a$, one can find open neighborhoods $\mathcal{U}(\omega)\subset Q$ and $\mathcal{V}(\omega)\subset Q^*$ which can be connected by the flow $F_{t,t_a}$, i.e., given $q_a \in \mathcal{U}(\omega)$ and $p \in \mathcal{V}(\omega)$, there exists a pathwise unique solution $(\bar q(\omega,\tau), \bar p(\omega,\tau))$ such that $\bar q(\omega,t_a)=q_a$ and $\bar p(\omega,t)=p$. With this assumption, \eqref{eq:Stochastic Hamilton-Jacobi equation along a path} can be reformulated as the full-blown stochastic PDE \eqref{eq:Type II Stochastic Hamilton-Jacobi equation}.\\
\end{proof}

\noindent
{\bf Remark:} Similar stochastic Hamilton-Jacobi equations were introduced in \cite{WangPHD}, \cite{Wang2014}, where they were used for constructing stochastic symplectic integrators by considering series expansions of generating functions in terms of multiple Stratonovich integrals. This was a direct generalization of a similar technique for deterministic Hamiltonian systems (see \cite{HLWGeometric}). In this work we explore the generalized Galerkin framework for constructing approximations of the generating function $S(q_a,p_b)$ in \eqref{eq:Equations generating the flow of the Hamiltonian system} by using its variational characterization \eqref{eq:Definition of the generating function}. Our approach is a stochastic generalization of the techniques proposed in \cite{LeokZhang} and \cite{OberBlobaum2015} for deterministic Hamiltonian and Lagrangian systems.

\subsection{Stochastic Noether's theorem}
\label{sec:Stochastic Noether's theorem}

Let a Lie group $G$ act on $Q$ by the left action $\Phi:G \times Q \longrightarrow Q$. The Lie group $G$ then acts on $TQ$ and $T^*Q$ by the tangent $\Phi^{TQ}:G \times TQ \longrightarrow TQ$ and cotangent $\Phi^{T^*Q}:G \times T^*Q \longrightarrow T^*Q$ lift actions, respectively, given in coordinates by the formulas (see \cite{HolmGMS}, \cite{MarsdenRatiuSymmetry})

\begin{align}
\label{eq:Tangent and cotangent lift actions}
\Phi^{TQ}_g(q,\dot q) & \equiv \Phi^{TQ}\big(g,(q,\dot q)\big) =\bigg( \Phi^i_g(q),\frac{\partial \Phi^i_g}{\partial q^j}(q) \dot q^j \bigg), \nonumber \\
\Phi^{T^*Q}_g(q,p) & \equiv \Phi^{T^*Q}\big(g,(q,p)\big) =\bigg( \Phi^i_g(q),p_j \frac{\partial \Phi^j_{g^{-1}}}{\partial q^i}\big(\Phi_g(q)\big) \bigg),
\end{align}

\noindent
where $i,j=1,\ldots,N$ and summation is implied over repeated indices. Let $\mathfrak{g}$ denote the Lie algebra of $G$ and $\exp: \mathfrak{g} \longrightarrow G$ the exponential map (see \cite{HolmGMS}, \cite{MarsdenRatiuSymmetry}). Each element $\xi \in \mathfrak{g}$ defines the infinitesimal generators $\xi_Q$, $\xi_{TQ}$, and $\xi_{T^*Q}$, which are vector fields on $Q$, $TQ$, and $T^*Q$, respectively, given by

\begin{align}
\label{eq:Infinitesimal generators}
\xi_Q(q) = \frac{d}{d \lambda} \bigg|_{\lambda =0} \Phi_{\exp \lambda \xi}(q), \qquad \xi_{TQ}(q, \dot q) = \frac{d}{d \lambda} \bigg|_{\lambda =0} \Phi^{TQ}_{\exp \lambda \xi}(q,\dot q), \qquad \xi_{T^*Q}(q, p) = \frac{d}{d \lambda} \bigg|_{\lambda =0} \Phi^{T^*Q}_{\exp \lambda \xi}(q,p).
\end{align}

\noindent
The momentum map $J: T^*Q \longrightarrow \mathfrak{g}^*$ associated with the action $\Phi^{T^*Q}$ is defined as the mapping such that for all $\xi \in \mathfrak{g}$ the function $J_\xi: T^*Q \ni (q,p) \longrightarrow \langle J(q,p),\xi \rangle \in \mathbb{R}$ is the Hamiltonian for the infinitesimal generator $\xi_{T^*Q}$, i.e.,

\begin{align}
\label{eq:Momentum map definition}
\xi^q_{T^*Q} = \frac{\partial J_\xi}{\partial p}, \qquad \xi^p_{T^*Q} = -\frac{\partial J_\xi}{\partial q},
\end{align}

\noindent
where $\xi_{T^*Q}(q,p) = \big(q,p,\xi^q_{T^*Q}(q,p),\xi^p_{T^*Q}(q,p)\big)$. The momentum map $J$ can be explicitly expressed as (see \cite{HolmGMS}, \cite{MarsdenRatiuSymmetry})

\begin{align}
\label{eq:Momentum map formula}
J_\xi(q,p) = p\cdot \xi_Q(q).
\end{align}

Noether's theorem for deterministic Hamiltonian systems relates symmetries of the Hamiltonian to quantities preserved by the flow of the system. It turns out that this result carries over to the stochastic case, as well. A stochastic version of Noether's theorem was proved in \cite{Bismut} and \cite{LaCa-Or2008}. For completeness, and for the benefit of the reader, below we state and provide a less involved proof of Noether's theorem for stochastic Hamiltonian systems.

\begin{thm}[{\bf Stochastic Noether's theorem}]
Suppose that the Hamiltonians $H:T^*Q \longrightarrow \mathbb{R}$ and $h:T^*Q \longrightarrow \mathbb{R}$ are invariant with respect to the cotangent lift action $\Phi^{T^*Q}:G \times T^*Q \longrightarrow T^*Q$ of the Lie group $G$, that is,

\begin{align}
\label{eq:Invariance of the Hamiltonians}
H \circ \Phi^{T^*Q}_g = H, \qquad \qquad h \circ \Phi^{T^*Q}_g = h,
\end{align}

\noindent
for all $g \in G$. Then the cotangent lift momentum map $J:T^*Q \longrightarrow \mathfrak{g}^*$ associated with $\Phi^{T^*Q}$ is almost surely preserved along the solutions of the stochastic Hamiltonian system \eqref{eq: Stochastic Hamiltonian system}.
\end{thm}

\begin{proof}
Equation \eqref{eq:Invariance of the Hamiltonians} implies that the Hamiltonians are infinitesimally invariant with respect to the action of $G$, that is, for all $\xi \in \mathfrak{g}$ we have 

\begin{align}
\label{eq:Infinitesimal invariance of the Hamiltonians}
dH\cdot \xi_{T^*Q} = 0, \qquad \qquad dh\cdot \xi_{T^*Q} = 0,
\end{align}

\noindent
where $dH$ and $dh$ denote differentials with respect to the variables $q$ and $p$. Let $(q(t),p(t))$ be a solution of \eqref{eq: Stochastic Hamiltonian system} and consider the stochastic process $J_\xi(q(t),p(t))$, where $\xi \in \mathfrak{g}$ is arbitrary. Using the rules of Stratonovich calculus we can calculate the stochastic differential

\begin{align}
\label{eq:Stochastic differential of J}
dJ_\xi\big(q(t),p(t)\big)& =\frac{\partial J_\xi}{\partial q}(q(t),p(t))\circ dq + \frac{\partial J_\xi}{\partial p}(q(t),p(t))\circ dp \nonumber \\
                         & =-\bigg( \frac{\partial H}{\partial q} \xi^q_{T^*Q} + \frac{\partial H}{\partial p} \xi^p_{T^*Q}  \bigg) \,dt -\bigg( \frac{\partial h}{\partial q} \xi^q_{T^*Q} + \frac{\partial h}{\partial p} \xi^p_{T^*Q}  \bigg)\circ dW \nonumber \\
                         & = -(dH\cdot \xi_{T^*Q})\,dt - (dh\cdot \xi_{T^*Q})\circ dW = 0,
\end{align}

\noindent
where we used \eqref{eq: Stochastic Hamiltonian system}, \eqref{eq:Momentum map definition}, and \eqref{eq:Infinitesimal invariance of the Hamiltonians}. Therefore, $J_\xi\big(q(t),p(t)\big) = \text{const}$ almost surely for all $\xi \in \mathfrak{g}$, which completes the proof.\\
\end{proof}

\section{Stochastic Galerkin Hamiltonian Variational Integrators}
\label{sec:Stochastic Galerkin Hamiltonian Variational Integrators}

If the converse to Theorem~\ref{thm:Stochastic Variational Principle} holds, then the generating function $S(q_a,p_b)$ defined in \eqref{eq:Definition of the generating function} could be equivalently characterized by

\begin{equation}
\label{eq:Variational definition of the generating function}
S(q_a,p_b) = \ext_{  \substack{(q,p)\in C([t_a,t_b])\\q(t_a)=q_a,\,\, p(t_b)=p_b}} \mathcal{B}\big[q(\cdot), p(\cdot) \big],
\end{equation}

\noindent
where the extremum is taken pointwise in the probability space $\Omega$. This characterization allows us to construct stochastic Galerkin variational integrators by choosing a finite dimensional subspace of $C([t_a,t_b])$ and a quadrature rule for approximating the integrals in the action functional $\mathcal{B}$. Galerkin variational integrators for deterministic systems were first introduced in \cite{MarsdenWestVarInt}, and further developed in \cite{HallLeokSpectral}, \cite{LeokShingel}, \cite{LeokZhang}, \cite{OberBlobaum2016}, and \cite{OberBlobaum2015}. In the remainder of the paper, we will generalize these ideas to the stochastic case.

\subsection{Construction of the integrator}
\label{sec:Construction of the integrator}

Suppose we would like to solve \eqref{eq: Stochastic Hamiltonian system} on the interval $[0,T]$ with the initial conditions $(q_0,p_0)\in T^*Q$. Consider the discrete set of times $t_k = k\cdot\Delta t$ for $k=0,1,\ldots,K$, where $\Delta t = T/K$ is the time step. In order to determine the discrete curve $\{(q_k,p_k)\}_{k=0,\ldots,K}$ that approximates the exact solution of \eqref{eq: Stochastic Hamiltonian system} at times $t_k$ we need to construct an approximation of the exact stochastic flow $F_{t_{k+1},t_k}$ on each interval $[t_k,t_{k+1}]$, so that $(q_{k+1},p_{k+1}) \approx F_{t_{k+1},t_k}(q_k,p_k)$. Let us consider the space $C^s([t_k,t_{k+1}])\subset C([t_k,t_{k+1}])$ defined as

\begin{equation}
\label{eq:Definition of the Cs space}
C^s([t_k, t_{k+1}]) = \big\{ (q,p)\in C([t_k, t_{k+1}])  \, \big| \, \text{$q$ is a polynomial of degree $s$} \big\}.
\end{equation}

\noindent
For convenience, we will express $q(t)$ in terms of Lagrange polynomials. Consider the control points $0=d_0<d_1<\ldots<d_s=1$ and let the corresponding Lagrange polynomials of degree $s$ be denoted by $l_{\mu,s}(\tau)$, that is, $l_{\mu,s}(d_\nu)=\delta_{\mu\nu}$. A polynomial trajectory $q_d(t;q^\mu)$ can then be expressed as

\begin{equation}
\label{eq:Polynomial q(t)}
q_d(t_k+\tau\Delta t; q^\mu) = \sum_{\mu = 0}^s q^\mu l_{\mu,s}(\tau), \qquad\qquad \dot q_d(t_k+\tau\Delta t; q^\mu) = \frac{1}{\Delta t}\sum_{\mu = 0}^s q^\mu \dot{l}_{\mu,s}(\tau),
\end{equation}

\noindent
where $q^\nu = q_d(t_k+d_\nu \Delta t; q^\mu)$ for $\nu=0,\ldots,s$ are the control values, $\dot q_d$ denotes the time derivative of $q_d$, and $\dot{l}_{\mu,s}$ denotes the derivative of the Lagrange polynomial $l_{\mu,s}$ with respect to its argument. The restriction of the action functional \eqref{eq:Stochastic action functional} to the space $C^s([t_k, t_{k+1}])$ takes the form

\begin{equation}
\label{eq:Restricted stochastic action functional}
\mathcal{B}^s\big[q_d(\cdot;q^\mu),p(\cdot) \big] = p(t_{k+1})q^s - \int_{t_k}^{t_{k+1}} \Big[ p(t)\dot q_d(t) - H\big(q_d(t),p(t)\big) \Big] \,dt + \int_{t_k}^{t_{k+1}} h\big(q_d(t),p(t)\big)\circ dW(t),
\end{equation}

\noindent
since for differentiable functions $dq_d(t) = \dot q_d(t)\,dt$, where for brevity $q_d(t) \equiv q_d(t;q^\mu)$. Next we approximate the integrals in \eqref{eq:Restricted stochastic action functional} using numerical quadrature rules $(\alpha_i, c_i)_{i=1}^r$ and $(\beta_i, c_i)_{i=1}^r$, where $0\leq c_1<\ldots<c_r\leq 1$ are the quadrature points, and $\alpha_i$, $\beta_i$ are the corresponding weights. At this point we only make a general assumption that for each $i$ we have $\alpha_i \not = 0$ or $\beta_i \not =0$. More specific examples will be presented in Section~\ref{sec:Examples}. The approximate action functional takes the form

\begin{align}
\label{eq:Approximate stochastic action functional}
\bar{\mathcal{B}}^s\big[q_d(\cdot;q^\mu),p(\cdot) \big] &= p(t_{k+1})q^s - \Delta t \sum_{i=1}^r \alpha_i \Big[ p(t_k+c_i\Delta t)\dot q_d(t_k+c_i\Delta t) - H\big(q_d(t_k+c_i\Delta t),p(t_k+c_i\Delta t)\big) \Big] \nonumber \\
&\phantom{= p(t_{k+1})q^s} + \Delta W \sum_{i=1}^r \beta_i  h\big(q_d(t_k+c_i\Delta t),p(t_k+c_i\Delta t)\big),
\end{align}

\noindent
where $\Delta W = W(t_{k+1})-W(t_k)$ is the increment of the Wiener process over the considered time interval and is a Gaussian random variable with zero mean and variance $\Delta t$. The way of approximating the Stratonovich integral above was inspired by the ideas presented in \cite{BouRabeeSVI}, \cite{Burrage2000}, \cite{MaDing2015}, \cite{MilsteinRepin2001}, and \cite{MilsteinRepin}. Note that since we only used $\Delta W = \int_{t_k}^{t_{k+1}}dW(t)$ in the above approximation, we can in general expect mean-square convergence of order 1.0 only. To obtain mean-square convergence of higher order we would also need to include higher-order multiple Stratonovich integrals, e.g., to achieve convergence of order 1.5 we would need to include terms involving $\Delta Z = \int_{t_k}^{t_{k+1}}\int_{t_k}^{t}dW(\xi)\,dt$ (see \cite{Burrage2000}, \cite{MilsteinRepin2001}, \cite{MilsteinRepin}). We can now approximate the generating function $S(q_k,p_{k+1})$ with the discrete Hamiltonian function $H^+_d(q_k,p_{k+1})$ defined as

\begin{align}
\label{eq:Discrete Hamiltonian}
H^+_d(q_k,p_{k+1}) &= \ext_{ \substack{q^1,\ldots,q^s \in Q \\ P_1, \ldots, P_r \in Q^* \\ q^0 = q_k} } \bigg\{ p_{k+1}q^s - \Delta t \sum_{i=1}^r \alpha_i \Big[ P_i \dot q_d(t_k+c_i\Delta t) - H\big(q_d(t_k+c_i\Delta t),P_i\big) \Big] \nonumber \\
&\phantom{= p_{k+1}q^saaaaaaaaaa} + \Delta W \sum_{i=1}^r \beta_i  h\big(q_d(t_k+c_i\Delta t),P_i\big) \bigg\},
\end{align}

\noindent
where we denoted $P_i\equiv p(t_k+c_i\Delta t)$. The numerical scheme $(q_k,p_k)\longrightarrow (q_{k+1},p_{k+1})$ is now implicitly generated by $H^+_d(q_k,p_{k+1})$ as in \eqref{eq:Equations generating the flow of the Hamiltonian system}:

\begin{equation}
\label{eq:Discrete flow generated by the discrete Hamiltonian}
q_{k+1} = D_2 H^+_d(q_k,p_{k+1}), \qquad\qquad p_k = D_1 H^+_d(q_k,p_{k+1}).
\end{equation}

\noindent
Equations \eqref{eq:Discrete Hamiltonian} and \eqref{eq:Discrete flow generated by the discrete Hamiltonian} can be written together as the following system:

\begin{subequations}
\label{eq:Stochastic Galerkin variational integrator}
\begin{align}
\label{eq:Stochastic Galerkin variational integrator 1}
-p_k&=\sum_{i=1}^r \alpha_i \Big[ P_i \dot l_{0,s}(c_i) - \Delta t \frac{\partial H}{\partial q}\big(t_k+c_i\Delta t\big) l_{0,s}(c_i) \Big] - \Delta W \sum_{i=1}^r \beta_i \frac{\partial h}{\partial q}\big(t_k+c_i\Delta t\big) l_{0,s}(c_i), \\
\label{eq:Stochastic Galerkin variational integrator 2}
0&=\sum_{i=1}^r \alpha_i \Big[ P_i \dot l_{\mu,s}(c_i) - \Delta t \frac{\partial H}{\partial q}\big(t_k+c_i\Delta t\big) l_{\mu,s}(c_i) \Big] - \Delta W \sum_{i=1}^r \beta_i \frac{\partial h}{\partial q}\big(t_k+c_i\Delta t\big) l_{\mu,s}(c_i),\\
\label{eq:Stochastic Galerkin variational integrator 3}
p_{k+1}&=\sum_{i=1}^r \alpha_i \Big[ P_i \dot l_{s,s}(c_i) - \Delta t \frac{\partial H}{\partial q}\big(t_k+c_i\Delta t\big) l_{s,s}(c_i) \Big] - \Delta W \sum_{i=1}^r \beta_i \frac{\partial h}{\partial q}\big(t_k+c_i\Delta t\big) l_{s,s}(c_i), \\
\label{eq:Stochastic Galerkin variational integrator 4}
\alpha_i \dot q_d&(t_k+c_i\Delta t) = \alpha_i \frac{\partial H}{\partial p}\big(t_k+c_i\Delta t\big) + \beta_i \frac{\Delta W}{\Delta t} \frac{\partial h}{\partial p}\big(t_k+c_i\Delta t\big),\\
\label{eq:Stochastic Galerkin variational integrator 5}
q_{k+1}&=q^s,
\end{align}
\end{subequations}

\noindent
where $\mu=1,\ldots,s-1$ in \eqref{eq:Stochastic Galerkin variational integrator 2}, $i=1,\ldots, r$ in \eqref{eq:Stochastic Galerkin variational integrator 4}, and for brevity we have introduced the notation 

\begin{equation*}
H(t_k+c_i\Delta t)\equiv H(q_d(t_k+c_i\Delta t),p(t_k+c_i\Delta t)) \textrm{\qquad (similarly for $h$).}
\end{equation*} 

\noindent
Equation \eqref{eq:Stochastic Galerkin variational integrator 1} corresponds to the second equation in \eqref{eq:Discrete flow generated by the discrete Hamiltonian}, equations \eqref{eq:Stochastic Galerkin variational integrator 2}, \eqref{eq:Stochastic Galerkin variational integrator 3} and \eqref{eq:Stochastic Galerkin variational integrator 4} correspond to extremizing \eqref{eq:Discrete Hamiltonian} with respect to $q^1, \ldots, q^s$, and $P_1,\ldots, P_r$, respectively, and finally \eqref{eq:Stochastic Galerkin variational integrator 5} is the first equation in \eqref{eq:Discrete flow generated by the discrete Hamiltonian}. Knowing $(q_k,p_k)$, the system \eqref{eq:Stochastic Galerkin variational integrator} allows us to solve for $(q_{k+1},p_{k+1})$: we first simultaneously solve \eqref{eq:Stochastic Galerkin variational integrator 1}, \eqref{eq:Stochastic Galerkin variational integrator 2} and \eqref{eq:Stochastic Galerkin variational integrator 4} ($(s+r)N$ equations) for $q^1, \ldots, q^s$ and $P_1,\ldots, P_r$ ($(s+r)N$ unknowns); then $q_{k+1}=q^s$ and \eqref{eq:Stochastic Galerkin variational integrator 3} is an explicit update rule for $p_{k+1}$. When $h\equiv 0$, then \eqref{eq:Stochastic Galerkin variational integrator} reduces to the deterministic Galerkin variational integrator discussed in \cite{OberBlobaum2015}. Note that depending on the choice of the Hamiltonians and quadrature rules, the system \eqref{eq:Stochastic Galerkin variational integrator} may be explicit, but in the general case it is implicit (see Section~\ref{sec:Examples}). One can use the Implicit Function Theorem to show that for sufficiently small $\Delta t$ and $|\Delta W|$ it will have a solution. However, since the increments $\Delta W$ are unbounded, for some values of $\Delta W$ solutions might not exist. To avoid problems with numerical implementations, if necessary, one can replace $\Delta W$ in \eqref{eq:Stochastic Galerkin variational integrator} with the truncated random variable $\overline{\Delta W}$ defined as

\begin{align}
\label{eq:Truncated Wiener increments}
\overline{\Delta W} =
\begin{cases}
A, & \text{if $\Delta W > A$},\\  
\Delta W, & \text{if $|\Delta W| \leq A$}, \\
 -A, & \text{if $\Delta W < -A$},
\end{cases}
\end{align}

\noindent
where $A>0$ is suitably chosen for the considered problem. See \cite{Burrage2004} and \cite{MilsteinRepin} for more details regarding schemes with truncated random increments and their convergence. Alternatively, one could employ the techniques presented in, e.g., \cite{Rossler2004}, \cite{Rossler2007}, and \cite{Wang2017}, where the unbounded increments $\Delta W$ have been replaced with discrete random variables.

Although the scheme \eqref{eq:Stochastic Galerkin variational integrator} formally appears to be a straightforward generalization of its deterministic counterpart, it should be emphasized that the main difference lies in the fact that the increments $\Delta W$ are random variables such that $E(\Delta W^2)=\Delta t$, which makes the convergence analysis significantly more challenging than in the deterministic case. The main difficulty is in the choice of the parameters $s$, $r$, $\alpha_i$, $\beta_i$, $c_i$, so that the resulting numerical scheme converges in some sense to the solutions of \eqref{eq: Stochastic Hamiltonian system}. The number of parameters and order conditions grows rapidly, when terms approximating multiple Stratonovich integrals are added (see Section~\ref{sec:Methods of order 3/2} and \cite{Burrage1996}, \cite{Burrage1998}, \cite{Burrage2000}, \cite{Burrage2004}). In Section~\ref{sec:Properties of stochastic Galerkin variational integrators} and Section~\ref{sec:Stochastic symplectic partitioned Runge-Kutta methods} we study the geometric properties of the family of schemes described by \eqref{eq:Stochastic Galerkin variational integrator}, whereas in Section~\ref{sec:Examples} and Section~\ref{sec:Convergence} we provide concrete choices of the coefficients that lead to convergent methods.

\subsection{Properties of stochastic Galerkin variational integrators}
\label{sec:Properties of stochastic Galerkin variational integrators}

The key features of variational integrators are their symplecticity and exact preservation of the discrete counterparts of conserved quantities (momentum maps) related to the symmetries of the Lagrangian or Hamiltonian (see \cite{MarsdenWestVarInt}). These properties carry over to the stochastic case, as was first demonstrated in \cite{BouRabeeSVI} for Lagrangian systems. In what follows, we will show that the stochastic Galerkin Hamiltonian variational integrators constructed in Section~\ref{sec:Construction of the integrator} also possess these properties.

\begin{thm}[{\bf Symplecticity of the discrete flow}] Let $F^+_{t_{k+1},t_k}:\Omega \times T^*Q \longrightarrow T^*Q$ be the dicrete stochastic flow implicitly defined by the discrete Hamiltonian $H^+_d$ as in \eqref{eq:Discrete flow generated by the discrete Hamiltonian}. Then $F^+_{t_{k+1},t_k}$ is almost surely symplectic, that is, 

\begin{equation}
\label{eq:Symplecticity of the discrete flow}
(F^+_{t_{k+1},t_k})^* \Omega_{T^*Q} = \Omega_{T^*Q},
\end{equation}

\noindent
where $\Omega_{T^*Q} = \sum_{i=1}^N dq^i \wedge dp^i$ is the canonical symplectic form on $T^*Q$.
\end{thm}

\begin{proof}
The proof follows immediately by observing that (see \cite{LeokZhang})

\begin{equation}
\label{eq:Proof of symplecticity}
0 = ddH^+(q_k,p_{k+1}) = \sum_{i=1}^N dq^i_{k+1} \wedge dp^i_{k+1} - \sum_{i=1}^N dq^i_{k} \wedge dp^i_{k} = (F^+_{t_{k+1},t_k})^* \Omega_{T^*Q} - \Omega_{T^*Q},
\end{equation}

\noindent
where $d$ in the above formula denotes the differential operator with respect to the variables $q$ and $p$ and is understood in the mean-square sense. The result holds almost surely, because equation \eqref{eq:Discrete flow generated by the discrete Hamiltonian} holds almost surely.\\
\end{proof}

The discrete counterpart of stochastic Noether's theorem readily generalizes from the corresponding theorem in the deterministic case.

\begin{thm}[{\bf Discrete stochastic Noether's theorem}] Let $\Phi^{T^*Q}$ be the cotangent lift action of the action $\Phi$ of the Lie group $G$ on the configuration space $Q$. If the generalized discrete stochastic Lagrangian $R_d(q_k, p_{k+1}) = p_{k+1}q_{k+1} - H^+_d(q_k,p_{k+1})$, where $q_{k+1}=D_2H^+_d(q_k,p_{k+1})$, is invariant under the action of $G$, that is,

\begin{equation}
\label{eq:Invariance of the generalized Lagrangian}
R_d \big( \Phi_g(q_k), \pi_{Q^*}\circ\Phi^{T^*Q}_g(q_{k+1},p_{k+1}) \big) = R_d(q_k, p_{k+1}), \quad \qquad \text{for all $g\in G$},
\end{equation}

\noindent
where $\pi_{Q^*}: Q \times Q^* \longrightarrow Q^*$ is the projection onto $Q^*$, then the cotangent lift momentum map $J$ associated with $\Phi^{T^*Q}$ is almost surely preserved, i.e., a.s. $J(q_{k+1},p_{k+1}) = J (q_{k},p_{k})$.
\end{thm}

\begin{proof}
See the proof of Theorem~4 in \cite{LeokZhang}. In our case the result holds almost surely, because equation \eqref{eq:Discrete flow generated by the discrete Hamiltonian} holds almost surely.\\
\end{proof}

\noindent
For applications, it is useful to know under what conditions the discrete Hamiltonian \eqref{eq:Discrete Hamiltonian} inherits the symmetry properties of the Hamiltonians $H$ and $h$. Not unexpectedly, this depends on the behavior of the interpolating polynomial \eqref{eq:Polynomial q(t)} under the group action. We say that the polynomial $q_d(t; q^\mu)$ is equivariant with respect to $G$ if for all $g \in G$ we have

\begin{equation}
\label{eq:Equivariance of the polynomial}
\Phi^{TQ}_g \Big( q_d(t; q^\mu), \dot q_d(t; q^\mu) \Big) = \Big( q_d\big(t; \Phi_g(q^\mu)\big), \dot q_d\big(t; \Phi_g(q^\mu)\big) \Big).
\end{equation}

\begin{thm} Suppose that the Hamiltonians $H:T^*Q \longrightarrow \mathbb{R}$ and $h:T^*Q \longrightarrow \mathbb{R}$ are invariant with respect to the cotangent lift action $\Phi^{T^*Q}:G \times T^*Q \longrightarrow T^*Q$ of the Lie group $G$, that is,

\begin{align}
\label{eq:Invariance of the Hamiltonians in discrete Nother's theorem}
H \circ \Phi^{T^*Q}_g = H, \qquad \qquad h \circ \Phi^{T^*Q}_g = h,
\end{align}

\noindent
for all $g \in G$, and suppose the interpolating polynomial $q_d(t; q^\mu)$ is equivariant with respect to $G$. Then the generalized discrete stochastic Lagrangian $R_d(q_k, p_{k+1}) = p_{k+1}q_{k+1} - H^+_d(q_k,p_{k+1})$ corresponding to the discrete Hamiltonian \eqref{eq:Discrete Hamiltonian}, where $q_{k+1}=D_2H^+_d(q_k,p_{k+1})$, is invariant with respect to the action of $G$.
\end{thm}

\begin{proof}
The proof is similar to the proofs of Lemma~3 in \cite{LeokZhang} and Theorem~3 in \cite{OberBlobaum2015}. Let us, however, carefully examine the actions of $G$ on $Q$, $TQ$, and $T^*Q$. Let $q_k \in Q$ and $p_{k+1} \in Q^*$, and let $q_{k+1}=D_2H^+_d(q_k,p_{k+1})$. First, note that for the stochastic discrete Hamiltonian \eqref{eq:Discrete Hamiltonian}, we have 

\begin{align}
\label{eq:Generalzed discrete Lagrangian for the Galerkin discrete Hamiltonian}
R(q_k, p_{k+1}) &= \ext_{ \substack{q^1,\ldots,q^s \in Q \\ P_1, \ldots, P_r \in Q^* \\ q^0 = q_k} } \bigg\{ \Delta t \sum_{i=1}^r \alpha_i \Big[ P_i \dot q_d(t_k+c_i\Delta t; q^\mu) - H\big(q_d(t_k+c_i\Delta t; q^\mu),P_i\big) \Big] \nonumber \\
&\phantom{= p_{k+1}q^saaaaaaaaaa} - \Delta W \sum_{i=1}^r \beta_i  h\big(q_d(t_k+c_i\Delta t; q^\mu),P_i\big) \bigg\},
\end{align}

\noindent
where we used \eqref{eq:Stochastic Galerkin variational integrator 5}. Consider $\tilde q_k = \Phi_g(q_k)$ and $(\tilde q_{k+1}, \tilde p_{k+1}) = \Phi^{T^*Q}_g(q_{k+1},p_{k+1})$ for $g \in G$, and calculate \eqref{eq:Generalzed discrete Lagrangian for the Galerkin discrete Hamiltonian} for the transformed values $\tilde q_k$ and $\tilde p_{k+1}$:

\begin{align}
R(\tilde q_k, \tilde p_{k+1}) &= \ext_{ \substack{\tilde q^1,\ldots,\tilde q^s \in Q \\ \tilde P_1, \ldots, \tilde P_r \in Q^* \\ \tilde q^0 = \tilde q_k} } \bigg\{ \Delta t \sum_{i=1}^r \alpha_i \Big[ \tilde P_i \dot q_d(t_k+c_i\Delta t; \tilde q^\mu) - H\big(q_d(t_k+c_i\Delta t; \tilde q^\mu),\tilde P_i\big) \Big] \nonumber \\
&\phantom{= p_{k+1}q^saaaaaaaaaa} - \Delta W \sum_{i=1}^r \beta_i  h\big(q_d(t_k+c_i\Delta t; \tilde q^\mu),\tilde P_i\big) \bigg\}.
\end{align}

\noindent
Let us perform a change of variables with respect to which we extremize. First, define $q^\mu = \Phi_{g^{-1}}(\tilde q^\mu)$, so that $\tilde q^\mu = \Phi_{g}(q^\mu)$ for $\mu=0,\ldots,s$. From \eqref{eq:Equivariance of the polynomial} we have $q_d(t_k+c_i\Delta t; \tilde q^\mu) = \Phi_g ( q_d(t_k+c_i\Delta t; q^\mu))$, which we use to define $P_i$ by $\big(q_d(t_k+c_i\Delta t; \tilde q^\mu),\tilde P_i\big) = \Phi^{T*Q}_g\big(q_d(t_k+c_i\Delta t; q^\mu),P_i\big)$ for $i=1,\ldots,r$. Since $\Phi_g$ and $\Phi^{T*Q}_g$ are bijective, extremization with respect to $q^\mu$ and $P_i$ is equivalent to extremization with respect to $\tilde q^\mu$ and $\tilde P_i$, and $\tilde q^0 = \tilde q_k$ implies  $q^0 = q_k$. Moreover, from \eqref{eq:Equivariance of the polynomial} and \eqref{eq:Tangent and cotangent lift actions} we have that $\tilde P_i \dot q_d(t_k+c_i\Delta t; \tilde q^\mu) = P_i \dot q_d(t_k+c_i\Delta t; q^\mu)$. Finally, the invariance of the Hamiltonians implies

\begin{align}
\label{eq:Transformed generalzed discrete Lagrangian is invariant}
R(\tilde q_k, \tilde p_{k+1}) &= \ext_{ \substack{q^1,\ldots,q^s \in Q \\ P_1, \ldots, P_r \in Q^* \\ q^0 = q_k} } \bigg\{ \Delta t \sum_{i=1}^r \alpha_i \Big[ P_i \dot q_d(t_k+c_i\Delta t; q^\mu) - H\big(q_d(t_k+c_i\Delta t; q^\mu),P_i\big) \Big] \nonumber \\
&\phantom{= p_{k+1}q^saaaaaaaaaa} - \Delta W \sum_{i=1}^r \beta_i  h\big(q_d(t_k+c_i\Delta t; q^\mu),P_i\big) \bigg\} = R(q_k, p_{k+1}),
\end{align}

\noindent
which completes the proof.\\
\end{proof}

\noindent
{\bf Remark:} One can easily verify that the interpolating polynomial \eqref{eq:Polynomial q(t)} is in particular equivariant with respect to linear actions (translations, rotations, etc.), therefore the stochastic Galerkin variational integrator \eqref{eq:Stochastic Galerkin variational integrator} preserves quadratic momentum maps (such as linear and angular momentum) related to linear symmetries of the Hamiltonians $H$ and $h$.

\subsection{Stochastic symplectic partitioned Runge-Kutta methods}
\label{sec:Stochastic symplectic partitioned Runge-Kutta methods}

A general class of stochastic Runge-Kutta methods for Stratonovich ordinary differential equations was introduced and analyzed in \cite{Burrage1996}, \cite{Burrage1998}, and \cite{Burrage2000}. These ideas were later used by Ma \& Ding \& Ding \cite{MaDing2012} and Ma \& Ding \cite{MaDing2015} to construct symplectic Runge-Kutta methods for stochastic Hamiltonian systems. An $s$-stage stochastic symplectic partitioned Runge-Kutta method for \eqref{eq: Stochastic Hamiltonian system} is defined in \cite{MaDing2015} by the following system:

\begin{subequations}
\label{eq:SPRK for stochastic Hamiltonian systems}
\begin{align}
\label{eq:SPRK for stochastic Hamiltonian systems 1}
Q_i &= q_k + \Delta t \sum_{j=1}^s a_{ij} \frac{\partial H}{\partial p}(Q_j,P_j) + \Delta W \sum_{j=1}^s b_{ij} \frac{\partial h}{\partial p}(Q_j,P_j), \quad \qquad i=1,\ldots,s,  \\
\label{eq:SPRK for stochastic Hamiltonian systems 2}
P_i &= p_k - \Delta t \sum_{j=1}^s \bar a_{ij} \frac{\partial H}{\partial q}(Q_j,P_j) - \Delta W \sum_{j=1}^s \bar b_{ij} \frac{\partial h}{\partial q}(Q_j,P_j), \quad \qquad i=1,\ldots,s, \\
\label{eq:SPRK for stochastic Hamiltonian systems 3}
q_{k+1} &= q_k + \Delta t \sum_{i=1}^s \alpha_i \frac{\partial H}{\partial p}(Q_i,P_i) + \Delta W \sum_{i=1}^s \beta_i \frac{\partial h}{\partial p}(Q_i,P_i),\\
\label{eq:SPRK for stochastic Hamiltonian systems 4}
p_{k+1} &= p_k - \Delta t \sum_{i=1}^s \alpha_i \frac{\partial H}{\partial q}(Q_i,P_i) - \Delta W \sum_{i=1}^s \beta_i \frac{\partial h}{\partial q}(Q_i,P_i),
\end{align}
\end{subequations}

\noindent
where $Q_i$ and $P_i$ for $i=1,\ldots,s$ are the position and momentum internal stages, respectively, and the coefficients of the method $a_{ij}$, $\bar a_{ij}$, $b_{ij}$, $\bar b_{ij}$, $\alpha_i$, $\beta_i$ satisfy the symplectic conditions

\begin{subequations}
\label{eq:Symplectic conditions for SPRK}
\begin{align}
\label{eq:Symplectic conditions for SPRK 1}
\alpha_i \bar a_{ij} + \alpha_j a_{ji} &= \alpha_i \alpha_j, \\
\label{eq:Symplectic conditions for SPRK 2}
\beta_i \bar a_{ij} + \alpha_j b_{ji} &= \beta_i \alpha_j, \\
\label{eq:Symplectic conditions for SPRK 3}
\alpha_i \bar b_{ij} + \beta_j a_{ji} &= \alpha_i \beta_j, \\
\label{eq:Symplectic conditions for SPRK 4}
\beta_i \bar b_{ij} + \beta_j b_{ji} &= \beta_i \beta_j,
\end{align}
\end{subequations}

\noindent
for $i,j=1,\ldots,s$. We now prove that in the special case when $r=s$, the stochastic Galerkin variational integrator \eqref{eq:Stochastic Galerkin variational integrator} is equivalent to the stochastic symplectic partitioned Runge-Kutta method \eqref{eq:SPRK for stochastic Hamiltonian systems}.

\begin{thm}
\label{thm:Stochastic Galerkin variational integrator as an SPRK method}
Let $r=s$ and let $\bar l_{i,s-1}(\tau)$ for $i=1,\ldots, s$ denote the Lagrange polynomials of degree $s-1$ associated with the quadrature points $0\leq c_1 < \ldots <c_s \leq 1$. Moreover, let the weights $\alpha_i$ be given by

\begin{equation}
\label{eq:Weights alpha_i in terms of Lagrange polynomials}
\alpha_i = \int_0^1 \bar l_{i,s-1}(\tau)\,d\tau,
\end{equation}

\noindent
and assume $\alpha_i \not = 0$ for $i=1,\ldots, s$. Then the stochastic Galerkin Hamiltonian variational integrator \eqref{eq:Stochastic Galerkin variational integrator} is equivalent to the stochastic partitioned Runge-Kutta method \eqref{eq:SPRK for stochastic Hamiltonian systems} with the coefficients

\begin{subequations}
\label{eq:Coefficients of the SPRK}
\begin{align}
\label{eq:Coefficients of the SPRK 1}
a_{ij}&=\int_0^{c_i} \bar l_{j,s-1}(\tau)\,d\tau, \\
\label{eq:Coefficients of the SPRK 2}
\bar a_{ij} &= \frac{\alpha_j(\alpha_i-a_{ji})}{\alpha_i}, \\
\label{eq:Coefficients of the SPRK 3}
b_{ij} &= \frac{\beta_j a_{ij}}{\alpha_j}, \\
\label{eq:Coefficients of the SPRK 4}
\bar b_{ij} &= \frac{\beta_j(\alpha_i-a_{ji})}{\alpha_i},
\end{align}
\end{subequations}

\noindent
for $i,j=1,\ldots,s$.

\end{thm}

\begin{proof}
The proof follows the main steps of the proof of Theorem~2.6.2 in \cite{MarsdenWestVarInt}. The time derivative $\dot q_d$ is a polynomial of degree $s-1$. Therefore, it can be uniquely expressed in terms of the Lagrange polynomials $\bar l_{j,s-1}(\tau)$ as

\begin{equation}
\label{eq:Time derivative of q_d in terms of the Lagrange polynomials}
\dot q_d(t_k+\tau \Delta t) = \sum_{j=1}^s \dot q_d(t_k+c_j \Delta t) \bar l_{j,s-1}(\tau).
\end{equation}

\noindent
Upon integrating with respect to time, we find

\begin{equation}
\label{eq:q_d in terms of the Lagrange polynomials}
q_d(t_k+\tau \Delta t) = q_k + \Delta t \sum_{j=1}^s \dot q_d(t_k+c_j \Delta t) \int_0^\tau \bar l_{j,s-1}(\xi)\,d\xi,
\end{equation}

\noindent
where we have used $q^0=q_k$. For $\tau=1$ this gives

\begin{equation}
\label{eq:q_d for tau=1 in terms of the Lagrange polynomials}
q_{k+1} = q_k + \Delta t \sum_{j=1}^s \alpha_j \dot q_d(t_k+c_j \Delta t),
\end{equation}

\noindent
where we have used $q^s = q_{k+1}$ and \eqref{eq:Weights alpha_i in terms of Lagrange polynomials}. Define the internal stages $Q_j\equiv q_d(t_k+c_j \Delta t)$. Then, upon using \eqref{eq:Stochastic Galerkin variational integrator 4}, equation \eqref{eq:q_d for tau=1 in terms of the Lagrange polynomials} becomes \eqref{eq:SPRK for stochastic Hamiltonian systems 3}. For $\tau=c_i$ equation \eqref{eq:q_d in terms of the Lagrange polynomials} gives

\begin{equation}
\label{eq:q_d for tau=c_i in terms of the Lagrange polynomials}
Q_i = q_k + \Delta t \sum_{j=1}^s a_{ij} \dot q_d(t_k+c_j \Delta t),
\end{equation}

\noindent
where $a_{ij}$ is defined by \eqref{eq:Coefficients of the SPRK 1}. Upon substituting \eqref{eq:Stochastic Galerkin variational integrator 4}, equation \eqref{eq:q_d for tau=c_i in terms of the Lagrange polynomials} becomes \eqref{eq:SPRK for stochastic Hamiltonian systems 1}, where $b_{ij}$ is defined by \eqref{eq:Coefficients of the SPRK 3}. Next, sum equations \eqref{eq:Stochastic Galerkin variational integrator 1}, \eqref{eq:Stochastic Galerkin variational integrator 2}, and \eqref{eq:Stochastic Galerkin variational integrator 3}. Noting that $\sum_{\mu=0}^s l_{\mu,s}(\tau) = 1$, this gives equation \eqref{eq:SPRK for stochastic Hamiltonian systems 4}. Finally, we note that for each $i=1,\ldots,s$ we have a unique decomposition

\begin{equation}
\int_0^\tau \bar l_{i,s-1}(\xi)\,d\xi - \alpha_i = \sum_{\mu=0}^s m_{i \mu} l_{\mu,s}(\tau),
\end{equation}

\noindent
since the left-hand side is a polynomial of degree $s$, and therefore it can be uniquely expressed as a linear combination of the Lagrange polynomials $l_{\mu,s}(\tau)$ with the coefficients $m_{i\mu}$. Evaluating this identity at $\tau=d_0=0$, $\tau=d_s=1$, and differentiating it with respect to $\tau$ yield the following three equations, respectively,

\begin{align}
\label{eq:Identities for the coefficients of the linear combination}
-\alpha_i &= \sum_{\mu=0}^s m_{i \mu} l_{\mu,s}(0) = m_{i0}, \nonumber \\
0&= \sum_{\mu=0}^s m_{i \mu} l_{\mu,s}(1) = m_{is}, \nonumber \\
\bar l_{i,s-1}(\tau) &= \sum_{\mu=0}^s m_{i \mu} \dot l_{\mu,s}(\tau).
\end{align}

\noindent
We form a linear combination of equations \eqref{eq:Stochastic Galerkin variational integrator 1}, \eqref{eq:Stochastic Galerkin variational integrator 2}, and \eqref{eq:Stochastic Galerkin variational integrator 3} with the coefficients $m_{j0}$, $m_{j\mu}$, and $m_{js}$, respectively. By using the identities \eqref{eq:Identities for the coefficients of the linear combination} and rearranging the terms, we obtain \eqref{eq:SPRK for stochastic Hamiltonian systems 4}, where $\bar a_{ij}$ and $\bar b_{ij}$ are defined by \eqref{eq:Coefficients of the SPRK 2} and \eqref{eq:Coefficients of the SPRK 4}, respectively. One can easily verify that the coefficients \eqref{eq:Coefficients of the SPRK} satisfy the conditions \eqref{eq:Symplectic conditions for SPRK}.\\
\end{proof}

\subsection{Examples}
\label{sec:Examples}

In the construction of the integrator \eqref{eq:Stochastic Galerkin variational integrator} we may choose the degree $s$ of the approximating polynomials and the quadrature rules $(\alpha_i, c_i)_{i=1}^r$ and $(\beta_i, c_i)_{i=1}^r$. In the deterministic case, the higher the degree of the polynomials and the higher the order of the quadrature rule, then the higher the order of convergence of the resulting integrator (see \cite{OberBlobaum2015}). In our case, however, as explained in Section~\ref{sec:Construction of the integrator}, we cannot in general achieve mean-square order of convergence higher than 1.0, because we only used $\Delta W$ in \eqref{eq:Approximate stochastic action functional}. Since the system \eqref{eq:Stochastic Galerkin variational integrator} requires solving $(s+r)N$ equations for $(s+r)N$ variables, from the computational point of view it makes sense to only consider methods with low values of $s$ and $r$. In this work we focus on the following classical numerical integration formulas (see \cite{HLWGeometric}, \cite{HWODE1}, \cite{HWODE2}):

\begin{itemize}
	\item Gauss-Legendre quadratures (Gau): midpoint rule ($r=1$), etc.
	\item Lobatto quadratures (Lob): trapezoidal rule ($r=2$), Simpson's rule ($r=3$), etc.
	\item Open trapezoidal rule (Otr; $r=2$)
	\item Milne's rule (Mil; $r=3$)
	\item Rectangle rule (Rec; $r=1$)---only in the case when $h=h(q)$.
\end{itemize}

In \cite{OberBlobaum2015} notation $PsNrQu$ was proposed to denote a Galerkin variational integrator based on polynomials of degree $s$ and a quadrature rule of order $u$ with $r$ quadrature points. We adopt a similar notation, keeping in mind that $u$ denotes the \emph{classical} order of the used quadrature rule---when the rule is applied to a stochastic integral, as in \eqref{eq:Approximate stochastic action functional}, its classical order is not attained in general. We also use a three-letter code to identify which integration formula is used. For example, $P2N2Q4Gau$ denotes the integrator defined by \eqref{eq:Stochastic Galerkin variational integrator} using polynomials of degree 2 and the Gauss-Legendre quadrature formula of classical order 4 with 2 quadrature points for both the Lebesgue and Stratonovich integrals in \eqref{eq:Approximate stochastic action functional}. If two different quadrature rules are used, we first write the rule applied to the Lebesgue integral, followed by the rule applied to the Stratonovich integral, e.g., $P1N1Q2GauN2Q2Lob$. Below we give several examples of integrators obtained by using polynomials of degree $s=1,2$ and the quadrature rules listed above.

\subsubsection{General Hamiltonian function $h(q,p)$}
\label{sec:General stochastic Hamiltonian}

For a general Hamiltonian $h=h(q,p)$, equation \eqref{eq:Stochastic Galerkin variational integrator 4}, which represents the discretization of the Legendre transform, needs to contain both $\partial H / \partial p$ and $\partial h / \partial p$ terms to correctly approximate the continuous system. Therefore, we only consider methods with $\alpha_i=\beta_i\not = 0$ for all $i=1,\ldots,r$. A few examples of interest are listed below.

\begin{enumerate}
	\item $P1N1Q2Gau$ (\emph{Stochastic midpoint method}) \\ Using the midpoint rule ($r=1$, $c_1=1/2$, $\alpha_1=\beta_1=1$) together with polynomials of degree $s=1$ gives a stochastic Runge-Kutta method \eqref{eq:SPRK for stochastic Hamiltonian systems} with $a_{11}=\bar a_{11}=b_{11}=\bar b_{11}=1/2$. Noting that $Q_1=(q_k+q_{k+1})/2$ and $P_1=(p_k+p_{k+1})/2$, this method can be written as
	
	\begin{align}
	\label{eq:Stochastic midpoint method}
	q_{k+1} &= q_k + \frac{\partial H}{\partial p} \bigg(\frac{q_k+q_{k+1}}{2},\frac{p_k+p_{k+1}}{2} \bigg)\Delta t 
	               + \frac{\partial h}{\partial p} \bigg(\frac{q_k+q_{k+1}}{2},\frac{p_k+p_{k+1}}{2} \bigg)\Delta W, \nonumber \\
	p_{k+1} &= p_k - \frac{\partial H}{\partial q} \bigg(\frac{q_k+q_{k+1}}{2},\frac{p_k+p_{k+1}}{2} \bigg)\Delta t 
	               - \frac{\partial h}{\partial q} \bigg(\frac{q_k+q_{k+1}}{2},\frac{p_k+p_{k+1}}{2} \bigg)\Delta W.
	\end{align}
	
	\noindent
	The stochastic midpoint method was considered in \cite{MaDing2015} and \cite{MilsteinRepin}. It is an implicit method and in general one has to solve $2N$ equations for $2N$ unknowns. However, if the Hamiltonians are separable, that is, $H(q,p)=T_0(p)+U_0(q)$ and $h(q,p)=T_1(p)+U_1(q)$, then $p_{k+1}$ from the second equation can be substituted into the first one. In that case only $N$ nonlinear equations have to be solved for $q_{k+1}$.
	
	\item $P2N2Q2Lob$ (\emph{Stochastic St{\"o}rmer-Verlet method}) \\ If the trapezoidal rule ($r=2$, $c_1=0$, $c_2=1$, $\alpha_1=\beta_1=1/2$, $\alpha_2=\beta_2=1/2$) is used with polynomials of degree $s=2$, we obtain another partitioned Runge-Kutta method \eqref{eq:SPRK for stochastic Hamiltonian systems} with $a_{11}=a_{12}=0$, $a_{21}=a_{22}=1/2$, $\bar a_{11}=\bar a_{21}=1/2$, $\bar a_{12}=\bar a_{22}=0$, $(b_{ij})=(a_{ij})$, $(\bar b_{ij})=(\bar a_{ij})$. Noting that $Q_1=q_k$, $Q_2=q_{k+1}$, and $P_1=P_2$, this method can be more efficiently written as
	
	\begin{align}
	\label{eq:Stochastic Stormer-Verlet method}
	P_1 &= p_k - \frac{1}{2} \frac{\partial H}{\partial q} \big(q_k, P_1 \big)\Delta t 
	           - \frac{1}{2} \frac{\partial h}{\partial q} \big(q_k, P_1 \big)\Delta W, \nonumber \\
	q_{k+1}&= q_k + \frac{1}{2} \frac{\partial H}{\partial p} \big(q_k, P_1 \big)\Delta t
	              + \frac{1}{2} \frac{\partial H}{\partial p} \big(q_{k+1}, P_1 \big)\Delta t
	              + \frac{1}{2} \frac{\partial h}{\partial p} \big(q_k, P_1 \big)\Delta W
	              + \frac{1}{2} \frac{\partial h}{\partial p} \big(q_{k+1}, P_1 \big)\Delta W, \nonumber \\
	p_{k+1}&= P_1 - \frac{1}{2} \frac{\partial H}{\partial q} \big(q_{k+1}, P_1 \big)\Delta t 
	              - \frac{1}{2} \frac{\partial h}{\partial q} \big(q_{k+1}, P_1 \big)\Delta W.
	\end{align}
	
	\noindent
	This method is a stochastic generalization of the St{\"o}rmer-Verlet method (see \cite{HLWGeometric}) and was considered in \cite{MaDing2015}. It is particularly efficient, because the first equation can be solved separately from the second one, and the last equation is an explicit update. Moreover, if the Hamiltonians are separable, this method becomes fully explicit.
	
	\item $P1N2Q2Lob$ (\emph{Stochastic trapezoidal method}) \\This integrator is based on polynomials of degree $s=1$ with control points $d_0=0$, $d_1=1$, and the trapezoidal rule. Equations \eqref{eq:Stochastic Galerkin variational integrator} take the form
	
	\begin{align}
	\label{eq:Stochastic trapezoidal method}
	p_k &= \frac{1}{2}(P_1+P_2) + \frac{1}{2} \frac{\partial H}{\partial q} \big(q_k, P_1 \big)\Delta t 
	                           + \frac{1}{2} \frac{\partial h}{\partial q} \big(q_k, P_1 \big)\Delta W, \nonumber \\
	p_{k+1} &= \frac{1}{2}(P_1+P_2) - \frac{1}{2} \frac{\partial H}{\partial q} \big(q_{k+1}, P_2 \big)\Delta t 
	                                - \frac{1}{2} \frac{\partial h}{\partial q} \big(q_{k+1}, P_2 \big)\Delta W, \nonumber \\
	q_{k+1} &= q_k + \frac{\partial H}{\partial p} \big(q_k,P_1 \big)\Delta t 
	               + \frac{\partial h}{\partial p} \big(q_k,P_1 \big)\Delta W, \nonumber \\
	q_{k+1} &= q_k + \frac{\partial H}{\partial p} \big(q_{k+1},P_2 \big)\Delta t 
	               + \frac{\partial h}{\partial p} \big(q_{k+1},P_2 \big)\Delta W.
	\end{align}
	
	\noindent
	This integrator is a stochastic generalization of the trapezoidal method for deterministic systems (see \cite{MarsdenWestVarInt}). One may easily verify that if the Hamiltonians are separable, that is, $H(q,p)=T_0(p)+U_0(q)$ and $h(q,p)=T_1(p)+U_1(q)$, then $P_1=P_2$ and \eqref{eq:Stochastic trapezoidal method} is equivalent to the St{\"o}rmer-Verlet method \eqref{eq:Stochastic Stormer-Verlet method} and is fully explicit. 
	
	\item $P1N3Q4Lob$ \\If we use Simpson's rule ($r=3$, $c_1=0$, $c_2=1/2$, $c_3=1$, $\alpha_1=1/6$, $\alpha_2=2/3$, $\alpha_3=1/6$, $\beta_i=\alpha_i$), the resulting integrator \eqref{eq:Stochastic Galerkin variational integrator} requires solving simultaneously $4N$ nonlinear equations, so it is computationally expensive in general. However, if the Hamiltonians $H$ and $h$ are separable, then \eqref{eq:Stochastic Galerkin variational integrator 4} implies $P_1=P_2=P_3$, and the integrator can be rewritten as
	
	\begin{align}
	\label{eq:Simpson's rule P1N3Q4Lob}
	q_{k+1}&= q_k + \frac{\partial T_0}{\partial p} \big(P_1 \big)\Delta t
	              + \frac{\partial T_1}{\partial p} \big(P_1 \big)\Delta W, \nonumber \\
	p_{k+1} &= P_1 - \frac{1}{3} \frac{\partial U_0}{\partial q} \bigg(\frac{q_k+q_{k+1}}{2}\bigg)\Delta t
	               - \frac{1}{6} \frac{\partial U_0}{\partial q} \big(q_{k+1}\big)\Delta t 
					       - \frac{1}{3} \frac{\partial U_1}{\partial q} \bigg(\frac{q_k+q_{k+1}}{2}\bigg)\Delta W
	               - \frac{1}{6} \frac{\partial U_1}{\partial q} \big(q_{k+1}\big)\Delta W,
	\end{align}
	
	\noindent
	where
	
	\begin{align}
		P_1 &= p_k - \frac{1}{6} \frac{\partial U_0}{\partial q} \big(q_k\big)\Delta t 
					   - \frac{1}{3} \frac{\partial U_0}{\partial q} \bigg(\frac{q_k+q_{k+1}}{2}\bigg)\Delta t
	           - \frac{1}{6} \frac{\partial U_1}{\partial q} \big(q_k\big)\Delta W 
					   - \frac{1}{3} \frac{\partial U_1}{\partial q} \bigg(\frac{q_k+q_{k+1}}{2}\bigg)\Delta W,
	\end{align}
	
	\noindent
	and $H(q,p)=T_0(p)+U_0(q)$ and $h(q,p)=T_1(p)+U_1(q)$. In this case only the first equation needs to be solved for $q_{k+1}$, and then the second equation is an explicit update.
	
	\item $P1N2Q2Otr$ \\Like the method \eqref{eq:Stochastic trapezoidal method}, this integrator is based on polynomials of degree $s=1$ with control points $d_0=0$, $d_1=1$, but uses the open trapezoidal rule ($r=2$, $c_1=1/3$, $c_2=2/3$, $\alpha_1=1/2$, $\alpha_2=1/2$, $\beta_i=\alpha_i$). Equations \eqref{eq:Stochastic Galerkin variational integrator} take the form
	
	\begin{align}
	\label{eq:Stochastic open trapezoidal method}
	p_k &= \frac{1}{2}(P_1+P_2) + \frac{1}{3} \frac{\partial H}{\partial q} \bigg(\frac{q_{k+1}+2q_k}{3}, P_1 \bigg)\Delta t 
	                            + \frac{1}{6} \frac{\partial H}{\partial q} \bigg(\frac{2q_{k+1}+q_k}{3}, P_2 \bigg)\Delta t \nonumber \\
	    &\phantom{=\frac{1}{2}(P_1+P_2)}  + \frac{1}{3} \frac{\partial h}{\partial q} \bigg(\frac{q_{k+1}+2q_k}{3}, P_1 \bigg)\Delta W
	                                      + \frac{1}{6} \frac{\partial h}{\partial q} \bigg(\frac{2q_{k+1}+q_k}{3}, P_2 \bigg)\Delta W, \nonumber \\
	p_{k+1} &= \frac{1}{2}(P_1+P_2) - \frac{1}{6} \frac{\partial H}{\partial q} \bigg(\frac{q_{k+1}+2q_k}{3}, P_1 \bigg)\Delta t 
	                            - \frac{1}{3} \frac{\partial H}{\partial q} \bigg(\frac{2q_{k+1}+q_k}{3}, P_2 \bigg)\Delta t \nonumber \\
	    &\phantom{=\frac{1}{2}(P_1+P_2)}  - \frac{1}{6} \frac{\partial h}{\partial q} \bigg(\frac{q_{k+1}+2q_k}{3}, P_1 \bigg)\Delta W
	                                      - \frac{1}{3} \frac{\partial h}{\partial q} \bigg(\frac{2q_{k+1}+q_k}{3}, P_2 \bigg)\Delta W, \nonumber \\
	q_{k+1} &= q_k + \frac{\partial H}{\partial p} \bigg(\frac{q_{k+1}+2q_k}{3}, P_1 \bigg)\Delta t 
	               + \frac{\partial h}{\partial p} \bigg(\frac{q_{k+1}+2q_k}{3}, P_1 \bigg)\Delta W, \nonumber \\
	q_{k+1} &= q_k + \frac{\partial H}{\partial p} \bigg(\frac{2q_{k+1}+q_k}{3}, P_2 \bigg)\Delta t 
	               + \frac{\partial h}{\partial p} \bigg(\frac{2q_{k+1}+q_k}{3}, P_2 \bigg)\Delta W.
	\end{align}
	
	\noindent
	In general one has to solve the first, third, and fourth equation simultaneously ($3N$ equations for $3N$ variables). In case of separable Hamiltonians we have $P_1=P_2$ and one only needs to solve $N$ nonlinear equations: $P_1$ can be explicitly calculated from the first equation and substituted into the third one, and the resulting nonlinear equation then has to be solved for $q_{k+1}$.
	
	\item $P2N2Q2Otr$ \\ If the open trapezoidal rule is used with polynomials of degree $s=2$, we obtain yet another partitioned Runge-Kutta method \eqref{eq:SPRK for stochastic Hamiltonian systems} with $a_{11}=\bar a_{22}=1/2$, $a_{12}=\bar a_{12}=-1/6$, $a_{21}=\bar a_{21}=2/3$, $a_{22}=\bar a_{11}=0$, $(b_{ij})=(a_{ij})$, $(\bar b_{ij})=(\bar a_{ij})$. Inspecting equations \eqref{eq:SPRK for stochastic Hamiltonian systems} we see that, for example, $Q_2$ is explicitly given in terms of $Q_1$ and $P_1$, therefore one only needs to solve $3N$ equations for the $3N$ variables $Q_1$, $P_1$, $P_2$, and the remaining equations are explicit updates. This method further simplifies for separable Hamiltonians $H$ and $h$: $Q_1$ and $Q_2$ are now explicitly given in terms of $P_1$ and $P_2$, and the nonlinear equation for $P_1$ can be solved separately from the nonlinear equation for $P_2$.
	
	\item $P1N3Q4Mil$ \\A method similar to \eqref{eq:Simpson's rule P1N3Q4Lob} is obtained if we use Milne's rule ($r=3$, $c_1=1/4$, $c_2=1/2$, $c_3=3/4$, $\alpha_1=2/3$, $\alpha_2=-1/3$, $\alpha_3=2/3$, $\beta_i=\alpha_i$) instead of Simpson's rule. The resulting integrator is also computationally expensive in general, but if the Hamiltonians $H$ and $h$ are separable, then \eqref{eq:Stochastic Galerkin variational integrator 4} implies $P_1=P_2=P_3$, and the integrator can be rewritten as
	
	\begin{align}
	\label{eq:Milne's rule P1N3Q4Mil}
	q_{k+1}= q_k &+ \frac{\partial T_0}{\partial p} \big(P_1 \big)\Delta t
	              + \frac{\partial T_1}{\partial p} \big(P_1 \big)\Delta W, \nonumber \\
	p_{k+1} = p_k &- \frac{2}{3} \frac{\partial U_0}{\partial q} \bigg(\frac{3 q_k+q_{k+1}}{4}\bigg)\Delta t 
					   + \frac{1}{3} \frac{\partial U_0}{\partial q} \bigg(\frac{q_k+q_{k+1}}{2}\bigg)\Delta t
					   - \frac{2}{3} \frac{\partial U_0}{\partial q} \bigg(\frac{q_k+3 q_{k+1}}{4}\bigg)\Delta t \nonumber \\
	           &- \frac{2}{3} \frac{\partial U_1}{\partial q} \bigg(\frac{3 q_k+q_{k+1}}{4}\bigg)\Delta W 
					   + \frac{1}{3} \frac{\partial U_1}{\partial q} \bigg(\frac{q_k+q_{k+1}}{2}\bigg)\Delta W
					   - \frac{2}{3} \frac{\partial U_1}{\partial q} \bigg(\frac{q_k+3 q_{k+1}}{4}\bigg)\Delta W,
	\end{align}
	
	\noindent
	where
	
	\begin{align}
		P_1 = p_k &- \frac{1}{2} \frac{\partial U_0}{\partial q} \bigg(\frac{3 q_k+q_{k+1}}{4}\bigg)\Delta t 
					   + \frac{1}{6} \frac{\partial U_0}{\partial q} \bigg(\frac{q_k+q_{k+1}}{2}\bigg)\Delta t
					   - \frac{1}{6} \frac{\partial U_0}{\partial q} \bigg(\frac{q_k+3 q_{k+1}}{4}\bigg)\Delta t \nonumber \\
	           &- \frac{1}{2} \frac{\partial U_1}{\partial q} \bigg(\frac{3 q_k+q_{k+1}}{4}\bigg)\Delta W 
					   + \frac{1}{6} \frac{\partial U_1}{\partial q} \bigg(\frac{q_k+q_{k+1}}{2}\bigg)\Delta W
					   - \frac{1}{6} \frac{\partial U_1}{\partial q} \bigg(\frac{q_k+3 q_{k+1}}{4}\bigg)\Delta W,
	\end{align}
	
	\noindent
	and $H(q,p)=T_0(p)+U_0(q)$ and $h(q,p)=T_1(p)+U_1(q)$. In this case only the first equation needs to be solved for $q_{k+1}$, and then the second equation is an explicit update.
	
\end{enumerate}

\subsubsection{Hamiltonian function $h(q)$ independent of momentum}
\label{sec:Stochastic Hamiltonian independent of momentum}

In case the Hamiltonian function $h=h(q)$ is independent of the momentum variable $p$, the term $\partial h / \partial p$ does not enter equation \eqref{eq:Stochastic Galerkin variational integrator 4}, and therefore we can allow a choice of quadrature rules such that $\alpha_i=0$ or $\beta_i=0$ for some $i$. If $\alpha_i=0$, however, the system \eqref{eq:Stochastic Galerkin variational integrator} becomes underdetermined, but at the same time the corresponding $P_i$ does not enter any of the remaining equations, therefore we can simply ignore it. To simplify the notation, let $i_1<\ldots<i_{\bar r}$ be the set of indices such that $\alpha_{i_m}\not = 0$, and denote $\bar \alpha_m\equiv \alpha_{i_m}$, $\bar c_m\equiv c_{i_m}$ for $m=1,\ldots,\bar r$. Similarly, let $j_1<\ldots<j_{\tilde r}$ be the set of indices such that $\beta_{j_m}\not = 0$, and denote $\tilde \beta_m\equiv \beta_{i_m}$, $\tilde c_m\equiv c_{j_m}$ for $m=1,\ldots,\tilde r$. In \eqref{eq:Stochastic Galerkin variational integrator} leave out the terms and equations corresponding to $\alpha_i=0$ or $\beta_i=0$, and replace $\alpha_i$, $\beta_i$, $c_i$ and $r$ by $\bar \alpha_i$, $\tilde \beta_i$, $\bar c_i$, $\tilde c_i$, $\bar r$ and $\tilde r$, accordingly. In other words, this is equivalent to using the quadrature rules $(\bar \alpha_i,\bar c_i)_{i=1}^{\bar r}$ and $(\tilde \beta_i,\tilde c_i)_{i=1}^{\tilde r}$ in \eqref{eq:Discrete Hamiltonian}. We then simultaneously solve \eqref{eq:Stochastic Galerkin variational integrator 1}, \eqref{eq:Stochastic Galerkin variational integrator 2} and \eqref{eq:Stochastic Galerkin variational integrator 4} ($(s+\bar r)N$ equations) for $q^1, \ldots, q^s$ and $P_1,\ldots, P_{\bar r}$ ($(s+\bar r)N$ unknowns). A few examples of such integrators are listed below.

\begin{enumerate}
	\item $P1N1Q1Rec$ (\emph{Stochastic symplectic Euler method}) \\The rectangle rule ($\bar r=1$, $\bar c_1=1$, $\bar \alpha_1=1$) does not yield a convergent numerical scheme in the general case, but when $h=h(q)$, the It\^o and Stratonovich interpretations of \eqref{eq: Stochastic Hamiltonian system} are equivalent, and the rectangle rule can be used to construct efficient integrators. In fact, applying the rectangle rule to both the Lebesgue and Stratonovich integrals and using polynomials of degree $s=1$ yield a method which can be written as 
	
	\begin{align}
	\label{eq:Stochastic symplectic Euler method}
	q_{k+1}&= q_k + \frac{\partial H}{\partial p} \big(q_{k+1}, p_k \big)\Delta t, \nonumber \\
	p_{k+1}&= p_k - \frac{\partial H}{\partial q} \big(q_{k+1}, p_k \big)\Delta t 
	              - \frac{\partial h}{\partial q} \big(q_{k+1}\big)\Delta W.
	\end{align}
	
	\noindent
	This method is a straightforward generalization of the symplectic Euler scheme (see \cite{HLWGeometric}, \cite{MarsdenWestVarInt}) and is particularly computationally efficient, as only the first equation needs to be solved for $q_{k+1}$, and then the second equation is an explicit update. Moreover, in case the Hamiltonian $H$ is separable, the method becomes explicit.

	\item $P1N1Q1RecN2Q2Lob$ \\The accuracy of the stochastic symplectic Euler scheme above can be improved by applying the trapezoidal rule to the Stratonovich integral instead of the rectangle rule. The resulting integrator takes the form
	
	\begin{align}
	\label{eq:Stochastic rectangle/trapezoidal method}
	q_{k+1} &= q_k + \frac{\partial H}{\partial p} \big(q_{k+1}, P_1 \big)\Delta t, \nonumber \\
	p_{k+1} &= p_k - \frac{\partial H}{\partial q} \big(q_{k+1},P_1 \big)\Delta t 
	               - \frac{1}{2} \frac{\partial h}{\partial q} \big(q_k\big)\Delta W
	               - \frac{1}{2} \frac{\partial h}{\partial q} \big(q_{k+1} \big)\Delta W,
	\end{align}
	
	\noindent
	where
	
	\begin{equation}
	P_1= p_k - \frac{1}{2} \frac{\partial h}{\partial q} \big(q_k \big) \Delta W.
	\end{equation}
	
	\noindent
	While having a similar computational cost, this method yields a more accurate solution than \eqref{eq:Stochastic symplectic Euler method} (see Section~\ref{sec:Numerical experiments} for numerical tests). Moreover, in case the Hamiltonian $H$ is separable, the method becomes explicit.
	
	\item $P1N1Q1RecN1Q2Gau$ \\Similarly, if we apply the midpoint rule instead of the trapezoidal rule, we obtain the following modification of the stochastic symplectic Euler method:
	
	\begin{align}
	\label{eq:Stochastic rectangle/midpoint method}
	q_{k+1} &= q_k + \frac{\partial H}{\partial p} \big(q_{k+1}, P_1 \big)\Delta t, \nonumber \\
	p_{k+1} &= p_k - \frac{\partial H}{\partial q} \big(q_{k+1},P_1 \big)\Delta t 
	               - \frac{\partial h}{\partial q} \bigg( \frac{q_k+q_{k+1}}{2}\bigg)\Delta W,
	\end{align}
	
	\noindent
	where
	
	\begin{equation}
	P_1= p_k - \frac{1}{2} \frac{\partial h}{\partial q} \bigg( \frac{q_k+q_{k+1}}{2}\bigg) \Delta W.
	\end{equation}
	
	\noindent
	This method demonstrates a similar performance as \eqref{eq:Stochastic rectangle/trapezoidal method} (see Section~\ref{sec:Numerical experiments} for numerical tests). It becomes explicit if the Hamiltonian $H$ is separable and the noise is additive, i.e., $\partial h / \partial q = \text{const}$.

	\item $P2N2Q2LobN1Q1Rec$ \\ A modification of the stochastic St{\"o}rmer-Verlet method \eqref{eq:Stochastic Stormer-Verlet method} is obtained if we use the rectangle rule to approximate the Stratonovich integral:
	
	\begin{align}
	\label{eq:Stochastic Stormer-Verlet/rectangle method}
	P_1 &= p_k - \frac{1}{2} \frac{\partial H}{\partial q} \big(q_k, P_1 \big)\Delta t, \nonumber \\
	q_{k+1}&= q_k + \frac{1}{2} \frac{\partial H}{\partial p} \big(q_k, P_1 \big)\Delta t
	              + \frac{1}{2} \frac{\partial H}{\partial p} \big(q_{k+1}, P_1 \big)\Delta t, \nonumber \\
	p_{k+1}&= P_1 - \frac{1}{2} \frac{\partial H}{\partial q} \big(q_{k+1}, P_1 \big)\Delta t 
	              - \frac{\partial h}{\partial q} \big(q_{k+1}\big)\Delta W.
	\end{align}
	
	\noindent
	This integrator has a similar computational cost as the stochastic St{\"o}rmer-Verlet method (see Section~\ref{sec:Numerical experiments}), but it yields a slightly less accurate solution (see Section~\ref{sec:Numerical experiments}). Moreover, in case the Hamiltonian $H$ is separable, the method becomes explicit.

	\item $P1N1Q2GauN2Q2Lob$ \\This integrator is a modification of the stochastic midpoint method \eqref{eq:Stochastic midpoint method}. We apply the midpoint rule ($\bar r=1$, $\bar c_1=1/2$, $\bar \alpha_1=1$) to the Lebesgue integral in \eqref{eq:Restricted stochastic action functional}, and the trapezoidal rule ($\tilde r=2$, $\tilde c_1=0$, $\tilde c_2=1$, $\tilde \beta_1=1/2$, $\tilde \beta_2=1/2$) to the Stratonovich integral. The resulting numerical scheme can be written as 
	
	\begin{align}
	\label{eq:Stochastic midpoint/trapezoidal method}
	q_{k+1} &= q_k + \frac{\partial H}{\partial p} \bigg(\frac{q_k+q_{k+1}}{2},P_1 \bigg)\Delta t, \nonumber \\
	p_{k+1} &= p_k - \frac{\partial H}{\partial q} \bigg(\frac{q_k+q_{k+1}}{2},P_1 \bigg)\Delta t 
	               - \frac{1}{2} \frac{\partial h}{\partial q} \big(q_k\big)\Delta W
	               - \frac{1}{2} \frac{\partial h}{\partial q} \big(q_{k+1} \big)\Delta W,
	\end{align}
	
	\noindent
	where
	
	\begin{equation}
	P_1=\frac{p_k+p_{k+1}}{2}+\frac{1}{4} \Delta W \bigg[\frac{\partial h}{\partial q} \big(q_{k+1} \big)-\frac{\partial h}{\partial q} \big(q_k \big) \bigg].
	\end{equation}
	
	\noindent
 This method is fully implicit, but similar to \eqref{eq:Stochastic midpoint method}, simplifies when the Hamiltonian $H$ is separable.
	
	\item $P1N2Q2LobN1Q2Gau$ \\If instead we apply the trapezoidal rule to the Lebesgue integral and the midpoint rule to the Stratonovich integral in \eqref{eq:Restricted stochastic action functional}, we obtain a modification of the stochastic trapezoidal rule \eqref{eq:Stochastic trapezoidal method}:
	
	\begin{align}
	\label{eq:Stochastic trapezoidal/midpoint method}
	p_k &= \frac{1}{2}(P_1+P_2) + \frac{1}{2} \frac{\partial H}{\partial q} \big(q_k, P_1 \big)\Delta t 
	                           + \frac{1}{2} \frac{\partial h}{\partial q} \bigg(\frac{q_k+q_{k+1}}{2} \bigg)\Delta W, \nonumber \\
	p_{k+1} &= \frac{1}{2}(P_1+P_2) - \frac{1}{2} \frac{\partial H}{\partial q} \big(q_{k+1}, P_2 \big)\Delta t 
	                                - \frac{1}{2} \frac{\partial h}{\partial q} \bigg(\frac{q_k+q_{k+1}}{2} \bigg)\Delta W, \nonumber \\
	q_{k+1} &= q_k + \frac{\partial H}{\partial p} \big(q_k,P_1 \big)\Delta t, \nonumber \\
	q_{k+1} &= q_k + \frac{\partial H}{\partial p} \big(q_{k+1},P_2 \big)\Delta t.
	\end{align}
	
	\noindent
	This method becomes explicit when the Hamiltonian $H$ is separable and the noise is additive, i.e., $\partial h / \partial q = \text{const}$.
	
\end{enumerate}

\subsection{Convergence}
\label{sec:Convergence}

Various criteria for convergence of stochastic schemes have been suggested in the literature (see \cite{KloedenPlatenSDE}, \cite{MilsteinBook}). Some criteria concentrate on pathwise approximations of the exact solutions (\emph{mean-square convergence, strong convergence}), while others focus on approximations of some functionals instead (\emph{weak convergence}). We are here primarily interested in mean-square convergence. Let $\bar z(t) = (\bar q(t),\bar p(t))$ be the exact solution to \eqref{eq: Stochastic Hamiltonian system} with the initial conditions $q_0$ and $p_0$, and let $z_k = (q_k,p_k)$ denote the numerical solution at time $t_k$ obtained by applying \eqref{eq:Stochastic Galerkin variational integrator} iteratively $k$ times with the constant time step $\Delta t$. The numerical solution is said to converge in the mean-square sense with global order $r$ if there exist $\delta>0$ and a constant $C>0$ such that for all $\Delta t \in (0,\delta)$ we have

\begin{equation}
\label{eq:Definition of mean-square convergence}
\sqrt{E(\| z_K-\bar z(T)\|^2 )} \leq C\Delta t^r,
\end{equation}

\noindent
where $T = K\Delta t$, as defined before, and $E$ denotes the expected value. In principle, in order to determine the mean-square order of convergence of the Galerkin variational integrator \eqref{eq:Stochastic Galerkin variational integrator} we need to calculate the power series expansions of $q_{k+1}$ and $p_{k+1}$ in terms of the powers of $\Delta t$ and $\Delta W$, and compare them to the Stratonovich-Taylor expansions for the exact solution $\bar q(t_k+\Delta t)$ and $\bar p(t_k+\Delta t)$ (see \cite{Burrage2000}, \cite{KloedenPlatenSDE}, \cite{MilsteinBook}). It is quite a tedious task to do in the general case, therefore we only discuss the examples presented in Section~\ref{sec:Examples}. For instance, in case of the stochastic trapezoidal method \eqref{eq:Stochastic trapezoidal method} we plug in series expansions for $P_1$, $P_2$, $q_{k+1}$ and $p_{k+1}$, and determine their coefficients by expanding the derivatives of the Hamiltonians into Taylor series around $(q_k,p_k)$ and comparing the terms corresponding to the like powers of $\Delta t$ and $\Delta W$. We find that

\begin{align}
\label{eq:Series expansion for the stochastic trapezoidal method}
q_{k+1} &= q_k + \frac{\partial H}{\partial p}\Delta t + \frac{\partial h}{\partial p}\Delta W + \frac{1}{2} \bigg(\frac{\partial^2 h}{\partial p \partial q} \frac{\partial h}{\partial p} - \frac{\partial^2 h}{\partial p^2} \frac{\partial h}{\partial q} \bigg) \Delta W^2+\ldots, \nonumber \\
p_{k+1}&= p_k - \frac{\partial H}{\partial q}\Delta t - \frac{\partial h}{\partial q}\Delta W - \frac{1}{2} \bigg(\frac{\partial^2 h}{\partial q^2} \frac{\partial h}{\partial p}  - \frac{\partial^2 h}{\partial q \partial p} \frac{\partial h}{\partial q}\bigg) \Delta W^2+\ldots,
\end{align} 

\noindent
where the derivatives of the Hamiltonians are evaluated at $(q_k,p_k)$. Let $\bar q(t;q_k,p_k)$ and $\bar p(t;q_k,p_k)$ denote the exact solution of \eqref{eq: Stochastic Hamiltonian system} such that $\bar q(t_k;q_k,p_k)=q_k$ and $\bar p(t_k;q_k,p_k)=p_k$. Using \eqref{eq: Stochastic Hamiltonian system} we calculate the Stratonovich-Taylor expansions for $\bar q(t_{k+1};q_k,p_k)$ and $\bar p(t_{k+1};q_k,p_k)$, and comparing them to \eqref{eq:Series expansion for the stochastic trapezoidal method} we find that

\begin{align}
\label{eq:Estimation of the remainders}
E\big(q_{k+1}-\bar q(t_{k+1};q_k,p_k)\big) &= O(\Delta t^2), \qquad \sqrt{E\big( \|q_{k+1}-\bar q(t_{k+1};q_k,p_k)\|^2 \big)} = O(\Delta t^{\frac{3}{2}}), \nonumber \\
E\big(p_{k+1}-\bar p(t_{k+1};q_k,p_k)\big) &= O(\Delta t^2), \qquad \sqrt{E\big( \|p_{k+1}-\bar p(t_{k+1};q_k,p_k)\|^2 \big)} = O(\Delta t^{\frac{3}{2}}).
\end{align}

\noindent
Using Theorem~1.1 from \cite{MilsteinBook}, we conclude that the stochastic trapezoidal method \eqref{eq:Stochastic trapezoidal method} has mean-square order of convergence $r=1$. In a similar fashion we prove that all methods presented in Section~\ref{sec:Examples} are convergent with mean-square order 1. We further verify these results numerically in Section~\ref{sec:Numerical convergence analysis}.

\paragraph{Remark.} For simplicity and clarity of the exposition, in this work we are mainly concerned with a one-dimensional noise in \eqref{eq: Stochastic Hamiltonian system}. However, all of the constructions and results presented in Section~\ref{sec:Variational principle for stochastic Hamiltonian systems} and Section~\ref{sec:Stochastic Galerkin Hamiltonian Variational Integrators} generalize in a straightforward manner, when a multidimensional noise $W^1, W^2, \ldots, W^M$, together with the corresponding Hamiltonian functions $h_1, h_2, \ldots, h_M$, is considered in \eqref{eq: Stochastic Hamiltonian system}, except that the integrators derived in Section~\ref{sec:Examples} in general do not attain mean-square order 1.0 of convergence, unless the noise is commutative. Indeed, if we repeat the procedure described above for the stochastic trapezoidal method, we will obtain the following power series expansions in terms of the powers of $\Delta t$ and $\Delta W^i$:

\begin{align}
\label{eq:Series expansion for the stochastic trapezoidal method with multidimensional noise}
q_{k+1} &= q_k + \frac{\partial H}{\partial p}\Delta t + \sum_{i=1}^M \frac{\partial h_i}{\partial p}\Delta W^i + \frac{1}{2} \sum_{i=1}^M \Gamma_{ii} (\Delta W^i)^2 + \frac{1}{2}\sum_{i=1}^M \sum_{\substack{j=1 \\ j \not = i}}^M \Gamma_{ij} \Delta W^i \Delta W^j + \ldots, \nonumber \\
p_{k+1}&= p_k - \frac{\partial H}{\partial q}\Delta t - \sum_{i=1}^M \frac{\partial h_i}{\partial q}\Delta W^i + \frac{1}{2} \sum_{i=1}^M \Lambda_{ii} (\Delta W^i)^2 + \frac{1}{2}\sum_{i=1}^M \sum_{\substack{j=1 \\ j \not = i}}^M \Lambda_{ij} \Delta W^i \Delta W^j +\ldots,
\end{align}

\noindent
where the vectors $\Gamma_{ij}$ and $\Lambda_{ij}$ for each $i,j=1,\ldots,M$ are defined as

\begin{align}
\label{eq: Gamma and Lambda terms definition}
\Gamma_{ij} =  \frac{\partial^2 h_j}{\partial p \partial q} \frac{\partial h_i}{\partial p} - \frac{\partial^2 h_j}{\partial p^2} \frac{\partial h_i}{\partial q}, \qquad \quad \Lambda_{ij} = -\frac{\partial^2 h_j}{\partial q^2} \frac{\partial h_i}{\partial p} + \frac{\partial^2 h_j}{\partial q \partial p} \frac{\partial h_i}{\partial q},
\end{align}

\noindent
and the derivatives of the Hamiltonians are evaluated at $(q_k, p_k)$. On the other hand, the Stratonovich-Taylor expansions for $\bar q(t_{k+1};q_k,p_k)$ and $\bar p(t_{k+1};q_k,p_k)$ read, respectively, 

\begin{align}
\label{eq:Stratonovich-Taylor expansion with multidimensional noise}
\bar q(t_{k+1};q_k,p_k) &= q_k + \frac{\partial H}{\partial p}\Delta t + \sum_{i=1}^M \frac{\partial h_i}{\partial p}\Delta W^i + \frac{1}{2} \sum_{i=1}^M \Gamma_{ii} (\Delta W^i)^2 + \sum_{i=1}^M \sum_{\substack{j=1 \\ j \not = i}}^M \Gamma_{ij} J_{ij} + \ldots, \nonumber \\
\bar p(t_{k+1};q_k,p_k)&= p_k - \frac{\partial H}{\partial q}\Delta t - \sum_{i=1}^M \frac{\partial h_i}{\partial q}\Delta W^i + \frac{1}{2} \sum_{i=1}^M \Lambda_{ii} (\Delta W^i)^2 + \sum_{i=1}^M \sum_{\substack{j=1 \\ j \not = i}}^M \Lambda_{ij} J_{ij} +\ldots,
\end{align}

\noindent
where $J_{ij} = \int_{t_k}^{t_{k+1}}\int_{t_k}^{t}dW^i(\tau)\circ dW^j(t)$ denotes a double Stratonovich integral. Comparing \eqref{eq: Gamma and Lambda terms definition} and \eqref{eq:Stratonovich-Taylor expansion with multidimensional noise}, we find that in the general case not all first order terms agree, and therefore we only have the local error estimates 

\begin{align}
\label{eq:Estimation of the remainders for multidimensional noise}
E\big(q_{k+1}-\bar q(t_{k+1};q_k,p_k)\big) &= O(\Delta t^{\frac{3}{2}}), \qquad \sqrt{E\big( \|q_{k+1}-\bar q(t_{k+1};q_k,p_k)\|^2 \big)} = O(\Delta t), \nonumber \\
E\big(p_{k+1}-\bar p(t_{k+1};q_k,p_k)\big) &= O(\Delta t^{\frac{3}{2}}), \qquad \sqrt{E\big( \|p_{k+1}-\bar p(t_{k+1};q_k,p_k)\|^2 \big)} = O(\Delta t).
\end{align}

\noindent
Theorem~1.1 from \cite{MilsteinBook} then implies that the stochastic trapezoidal method has mean-square order $1/2$. However, if the noise is commutative, that is, if the following conditions are satisfied

\begin{equation}
\label{eq: Commutation conditions}
\Gamma_{ij}=\Gamma_{ji}, \qquad \Lambda_{ij}=\Lambda_{ji}, \qquad \text{for all $i,j=1,\ldots,M$},
\end{equation}  

\noindent
then using the property $J_{ij}+J_{ji}=\Delta W^i \Delta W^j$ (see \cite{KloedenPlatenSDE}, \cite{MilsteinBook}), one can easily show

\begin{equation}
\label{eq: Jij terms from the Stratonovich-Taylor expansion}
\sum_{i=1}^M \sum_{\substack{j=1 \\ j \not = i}}^M \Gamma_{ij} J_{ij} = \frac{1}{2}\sum_{i=1}^M \sum_{\substack{j=1 \\ j \not = i}}^M \Gamma_{ij} \Delta W^i \Delta W^j, \qquad \quad \sum_{i=1}^M \sum_{\substack{j=1 \\ j \not = i}}^M \Lambda_{ij} J_{ij} = \frac{1}{2}\sum_{i=1}^M \sum_{\substack{j=1 \\ j \not = i}}^M \Lambda_{ij} \Delta W^i \Delta W^j.
\end{equation} 

\noindent
In that case all first-order terms in the expansions \eqref{eq:Series expansion for the stochastic trapezoidal method with multidimensional noise} and \eqref{eq:Stratonovich-Taylor expansion with multidimensional noise} agree, and we again have the local error estimates \eqref{eq:Estimation of the remainders}, meaning that the scheme has mean-square order $1.0$. Similar analysis holds for all the methods presented in Section~\ref{sec:Examples}. It should be noted that the commutation conditions \eqref{eq: Commutation conditions} hold for two important special cases: 

\begin{itemize}
\item Hamiltonian functions $h_i$ linear in $q$ and $p$ for all $i=1,\ldots,M$, i.e. additive noise
\item Hamiltonian functions $h_i$ simultaneously independent of one of the variables $q$ or $p$ for all $i=1,\ldots,M$
\end{itemize}

\noindent
The latter in particular means that the methods presented in Section~\ref{sec:Stochastic Hamiltonian independent of momentum} retain their mean-square order of convergence for multidimensional noises.

\subsection{Methods of order $3/2$}
\label{sec:Methods of order 3/2}

In order to construct stochastic Galerkin variational integrators of higher order one needs to include higher order terms in the discretization of the Stratonovich integral in \eqref{eq:Approximate stochastic action functional}. For example, a method of mean-square order 3/2 must include terms involving $\Delta Z = \int_{t_k}^{t_{k+1}}\int_{t_k}^{t}dW(\xi)\,dt$ (see \cite{Burrage2000}, \cite{MilsteinRepin2001}, \cite{MilsteinRepin}). Inspired by the theory presented in \cite{Burrage2000}, we can add extra terms to the discrete Hamiltonian \eqref{eq:Discrete Hamiltonian} and write it as

\begin{align}
\label{eq:Discrete Hamiltonian with dZ}
H^+_d(q_k,p_{k+1}) &= \ext_{ \substack{q^1,\ldots,q^s \in Q \\ P_1, \ldots, P_r \in Q^* \\ q^0 = q_k} } \bigg\{ p_{k+1}q^s - \Delta t \sum_{i=1}^r \alpha_i \Big[ P_i \dot q_d(t_k+c_i\Delta t) - H\big(q_d(t_k+c_i\Delta t),P_i\big) \Big] \nonumber \\
&\phantom{= p_{k+1}q^sa} + \Delta W \sum_{i=1}^r \beta_i  h\big(q_d(t_k+c_i\Delta t),P_i\big) + \frac{\Delta Z}{\Delta t} \sum_{i=1}^r \gamma_i  h\big(q_d(t_k+c_i\Delta t),P_i\big) \bigg\}.
\end{align}

\noindent
The random variables $\Delta W$ and $\Delta Z$ have a Gaussian joint distribution (see \cite{KloedenPlatenSDE}, \cite{MilsteinRepin}), and at each time step they can be simulated by two independent $\mathcal{N}(0,1)$-distributed random variables $\chi$ and $\eta$ as

\begin{equation}
\label{eq:dW and dZ}
\Delta W = \chi \sqrt{\Delta t}, \qquad \qquad \Delta Z = \frac{1}{2}\Delta t^{\frac{3}{2}} \Big(\chi + \frac{1}{\sqrt{3}} \eta \Big).
\end{equation}

\noindent
In order to achieve mean-square convergence of order 3/2 one needs to determine appropriate values for the parameters $s$, $r$, $\alpha_i$, $\beta_i$, $\gamma_i$, and $c_i$. However, we will not attempt to achieve this in the present work. Instead, we will show that some known stochastic symplectic integrators can be derived as stochastic Galerkin variational integrators.

Suppose the Hamiltonian is separable, i.e., $H(q,p)=T(p)+U(q)$, and the Hamiltonian function $h=h(q)$ does not depend on momentum. Consider the discrete Hamiltonian

\begin{align}
\label{eq:Discrete Hamiltonian with dZ for separable Hamiltonians}
H^+_d(q_k,p_{k+1}) &= \ext_{ \substack{q^1,\ldots,q^s \in Q \\ P_1, \ldots, P_r \in Q^* \\ q^0 = q_k} } \bigg\{ p_{k+1}q^s - \Delta t \sum_{i=1}^r \Big[ \bar \alpha_i P_i \dot q_d(t_k+c_i\Delta t) - \bar \alpha_i U\big(q_d(t_k+c_i\Delta t)\big) - \alpha_i T\big(P_i\big) \Big] \nonumber \\
&\phantom{= p_{k+1}q^sa} + \Delta W \sum_{i=1}^r \bar \beta_i  h\big(q_d(t_k+c_i\Delta t)\big) + \frac{\Delta Z}{\Delta t} \sum_{i=1}^r \bar \gamma_i  h\big(q_d(t_k+c_i\Delta t)\big) \bigg\},
\end{align}

\noindent
where different weights $\bar \alpha_i$ and $\alpha_i$ were applied to the potential $U(q)$ and kinetic $T(p)$ terms, respectively. Similar to \eqref{eq:Stochastic Galerkin variational integrator}, the corresponding stochastic variational integrator takes the form

\begin{align}
-p_k&=\sum_{i=1}^r \bar \alpha_i \Big[ P_i \dot l_{0,s}(c_i) - \Delta t \frac{\partial U}{\partial q}\big(t_k+c_i\Delta t\big) l_{0,s}(c_i) \Big] - \sum_{i=1}^r \Big(\bar \beta_i \Delta W + \bar \gamma_i\frac{\Delta Z}{\Delta t}\Big) \frac{\partial h}{\partial q}\big(t_k+c_i\Delta t\big) l_{0,s}(c_i), \nonumber \\
0&=\sum_{i=1}^r \bar \alpha_i \Big[ P_i \dot l_{\mu,s}(c_i) - \Delta t \frac{\partial U}{\partial q}\big(t_k+c_i\Delta t\big) l_{\mu,s}(c_i) \Big] - \sum_{i=1}^r \Big(\bar \beta_i \Delta W + \bar \gamma_i\frac{\Delta Z}{\Delta t}\Big) \frac{\partial h}{\partial q}\big(t_k+c_i\Delta t\big) l_{\mu,s}(c_i), \nonumber \\
p_{k+1}&=\sum_{i=1}^r \bar \alpha_i \Big[ P_i \dot l_{s,s}(c_i) - \Delta t \frac{\partial U}{\partial q}\big(t_k+c_i\Delta t\big) l_{s,s}(c_i) \Big] - \sum_{i=1}^r \Big(\bar \beta_i \Delta W + \bar \gamma_i\frac{\Delta Z}{\Delta t}\Big) \frac{\partial h}{\partial q}\big(t_k+c_i\Delta t\big) l_{s,s}(c_i), \nonumber \\
\label{eq:Stochastic Galerkin variational integrator with dZ}
\bar \alpha_i \dot q_d&(t_k+c_i\Delta t) = \alpha_i \frac{\partial T}{\partial p}\big(P_i\big),\\
q_{k+1}&=q^s, \nonumber
\end{align}

\noindent
where $\mu=1,\ldots,s-1$ in the second equation, and $i=1,\ldots, r$ in the fourth equation. In the special case when $r=s$ and

\begin{equation}
\label{eq:Weights bar alpha_i in terms of Lagrange polynomials}
\bar \alpha_i = \int_0^1 \bar l_{i,s-1}(\tau)\,d\tau, \quad \qquad i=1,\ldots,s,
\end{equation}

\noindent
we can show, similar to Theorem~\ref{thm:Stochastic Galerkin variational integrator as an SPRK method}, that the stochastic Galerkin variational integrator \eqref{eq:Stochastic Galerkin variational integrator with dZ} is equivalent to the stochastic partitioned Runge-Kutta method

\begin{align}
\label{eq:SPRK for separable stochastic Hamiltonian systems with dZ}
Q_i &= q_k + \Delta t \sum_{j=1}^s a_{ij} \frac{\partial T}{\partial p}(P_j), \phantom{   - \sum_{j=1}^s \Big( \bar b_{ij} \Delta W + \bar \lambda_{ij}\frac{\Delta Z}{\Delta t} \Big) \frac{\partial h}{\partial q}(Q_j),,  }  \quad \qquad i=1,\ldots,s,  \nonumber \\
P_i &= p_k - \Delta t \sum_{j=1}^s \bar a_{ij} \frac{\partial U}{\partial q}(Q_j) - \sum_{j=1}^s \Big( \bar b_{ij} \Delta W + \bar \lambda_{ij}\frac{\Delta Z}{\Delta t} \Big) \frac{\partial h}{\partial q}(Q_j), \quad \qquad i=1,\ldots,s, \nonumber \\
q_{k+1} &= q_k + \Delta t \sum_{i=1}^s \alpha_i \frac{\partial T}{\partial p}(P_i),\\
p_{k+1} &= p_k - \Delta t \sum_{i=1}^s \bar \alpha_i \frac{\partial U}{\partial q}(Q_i) - \sum_{i=1}^s \Big( \bar \beta_i \Delta W + \bar \gamma_i \frac{\Delta Z}{\Delta t} \Big) \frac{\partial h}{\partial q}(Q_i), \nonumber 
\end{align}

\noindent
with the coefficients

\begin{align}
\label{eq:Coefficients of the SPRK with dZ}
a_{ij}&=\frac{\alpha_j}{\bar \alpha_j}\int_0^{c_i} \bar l_{j,s-1}(\tau)\,d\tau,  & \bar a_{ij} &= \frac{\bar \alpha_j(\alpha_i-a_{ji})}{\alpha_i}, \nonumber \\
\bar b_{ij} &= \frac{\bar \beta_j(\alpha_i-a_{ji})}{\alpha_i}, & \bar \lambda_{ij} &= \frac{\bar \gamma_j(\alpha_i-a_{ji})}{\alpha_i}, \qquad \qquad i,j=1,\ldots,s,
\end{align}

\noindent
where we assume $\alpha_i \not = 0$ and $\bar \alpha_i \not = 0$ for all $i$. Partitioned Runge-Kutta methods of type \eqref{eq:SPRK for separable stochastic Hamiltonian systems with dZ} were considered in \cite{MilsteinRepin}. In particular, it was shown that for $s=2$ the choice of the coefficients 

\begin{align}
\label{eq:Coefficients of Milstein's SPRK}
\alpha_1&=2/3,  & \alpha_2&=1/3, &  \bar \alpha_1&=1/4,  &  \bar \alpha_2&=3/4,  &  \bar \beta_1&= -1/2,  &  \bar \beta_2&=3/2,  &  \bar \gamma_1&=3/2  &  \bar \gamma_2&=-3/2, \nonumber \\
a_{11}&= 0,  &  a_{12}&=0,  &  \bar a_{11}&=1/4,  &  \bar a_{12}&=0,  &  \bar b_{11}&=-1/2,  &  \bar b_{12}&=0,  &  \bar \lambda_{11}&=3/2,  &  \bar \lambda_{12}&=0, \nonumber\\
a_{21}&= 2/3,  &  a_{22}&=0,  &  \bar a_{21}&=1/4,  &  \bar a_{22}&=3/4,  &  \bar b_{21}&=-1/2,  &  \bar b_{22}&=3/2,  &  \bar \lambda_{21}&=3/2,  &  \bar \lambda_{22}&=-3/2,
\end{align}

\noindent
gives a method of mean-square order 3/2 (see Theorem~4.3 in \cite{MilsteinRepin}).

\section{Numerical experiments}
\label{sec:Numerical experiments}

In this section we present the results of our numerical experiments. We verify numerically the convergence results from Section~\ref{sec:Convergence} and investigate the conservation properties of our integrators. In particular, we show that our stochastic variational integrators demonstrate superior behavior in long-time simulations compared to some popular general purpose non-symplectic stochastic algorithms. 

\subsection{Numerical convergence analysis}
\label{sec:Numerical convergence analysis}

\subsubsection{Kubo oscillator}
\label{sec:Kubo oscillator convergence}
In order to test the convergence of the numerical algorithms from Section~\ref{sec:General stochastic Hamiltonian} we performed computations for the Kubo oscillator, which is defined by $H(q,p)=p^2/2+q^2/2$ and $h(q,p)=\beta(p^2/2+q^2/2)$, where $\beta$ is the noise intensity (see \cite{MilsteinRepin}). The Kubo oscillator is used in the theory of magnetic resonance and laser physics. The exact solution is given by 

\begin{equation}
\label{eq:Kubo oscillator---exact solution}
\bar q(t)=p_0 \sin(t+\beta W(t)) + q_0 \cos(t+\beta W(t)), \qquad\quad \bar p(t)=p_0 \cos(t+\beta W(t)) - q_0 \sin(t+\beta W(t)),
\end{equation}

\noindent
where $q_0$ and $p_0$ are the initial conditions. Simulations with the initial conditions $q_0=0$, $p_0=1$ and the noise intensity $\beta=0.1$ were carried out until the time $T=3.2$ for a number of time steps $\Delta t = 0.000625, 0.00125, 0.0025, 0.005, 0.01, 0.02$. In each case 2000 sample paths were generated. Let $z_{\Delta t}(t) = (q_{\Delta t}(t), p_{\Delta t}(t) )$ denote the numerical solution. We used the exact solution \eqref{eq:Kubo oscillator---exact solution} as a reference for computing the mean-square error $\sqrt{E(|z_{\Delta t}(T)-\bar z(T)|^2)}$, where $\bar z(t) = (\bar q(t), \bar p(t) )$. The dependence of this error on the time step $\Delta t$ is depicted in Figure~\ref{fig: Convergence plot for the Kubo oscillator}. We verified that our algorithms have mean-square order of convergence $1.0$. The integrators $P1N3Q4Lob$, $P1N3Q4Mil$, $P1N2Q2Lob$ (stochastic trapezoidal method), and $P2N2Q2Lob$ (stochastic St{\"o}rmer-Verlet method) turned out to be the most accurate, with the latter two having least computational cost.

\begin{figure}
	\centering
		\includegraphics[width=\textwidth]{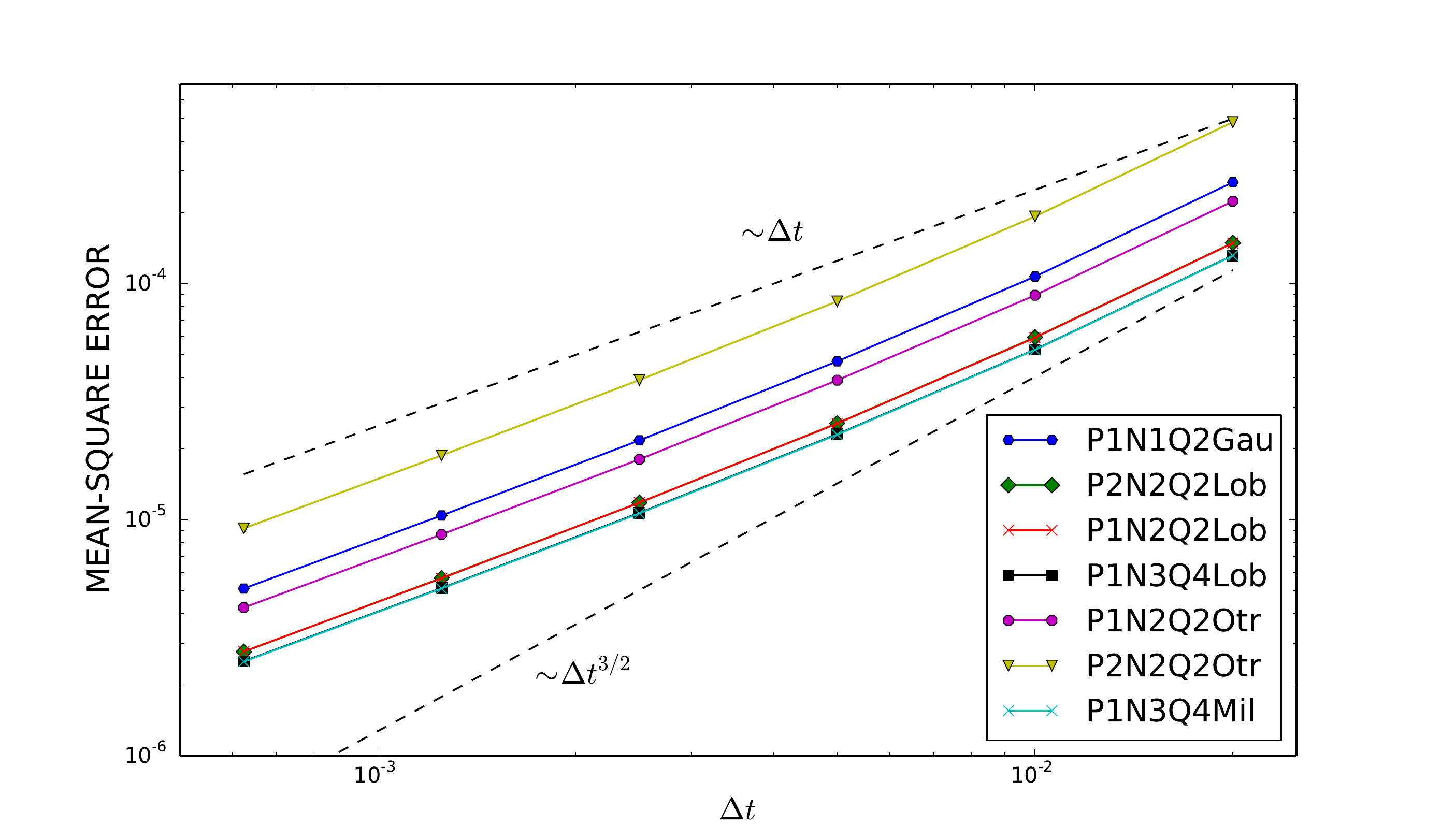}
		\caption{ The mean-square error at the time $T=3.2$ as a function of the time step $\Delta t$ for the simulations of the Kubo oscillator with the initial conditions $q_0=0$, $p_0=1$ and the noise intensity $\beta=0.1$. In each case 2000 sample paths were generated. The tested integrators proved to be convergent of order 1.0 in the mean-square sense. Note that the plots for $P2N2Q2Lob$ and $P1N2Q2Lob$, as well as for $P1N3Q4Lob$ and $P1N3Q4Mil$, overlap very closely.}
		\label{fig: Convergence plot for the Kubo oscillator}
\end{figure}

\subsubsection{Synchrotron oscillations of particles in storage rings}
\label{sec:Synchrotron oscillations of particles in storage rings}
We carried out a similar test for the numerical schemes from Section~\ref{sec:Stochastic Hamiltonian independent of momentum}. We performed computations for the stochastic Hamiltonian system defined by $H(q,p)=p^2/2-\cos q$ and $h(q)=\beta \sin q$, where $\beta$ is the noise intensity. Systems of this type are used for modeling synchrotron oscillations of a particle in a storage ring. Due to fluctuating electromagnetic fields, a particle performs stochastically perturbed oscillations with respect to a reference particle which travels with fixed energy along the design orbit of the accelerator; in this description $p$ corresponds to the energy deviation of the particle from the reference particle, and $q$ measures the longitudinal phase difference of both particles (see \cite{DomeAccelerators}, \cite{SeesselbergParticleStorageRings} for more details). Simulations with the initial conditions $q_0=0$, $p_0=1$ and the noise intensity $\beta=0.1$ were carried out until the time $T=3.2$ for a number of time steps $\Delta t = 0.01, 0.02, 0.04, 0.08, 0.16, 0.32, 0.64$. In each case 2000 sample paths were generated. The mean-square error was calculated with respect to a high-precision reference solution generated using the order 3/2 strong Taylor scheme (see \cite{KloedenPlatenSDE}, Chapter~10.4) with a very fine time step $\Delta t = 2\cdot 10^{-6}$. The dependence of this error on the time step $\Delta t$ is depicted in Figure~\ref{fig: Convergence plot for the synchrotron oscillations}. We verified that our algorithms have mean-square order of convergence $1.0$.

\begin{figure}
	\centering
		\includegraphics[width=\textwidth]{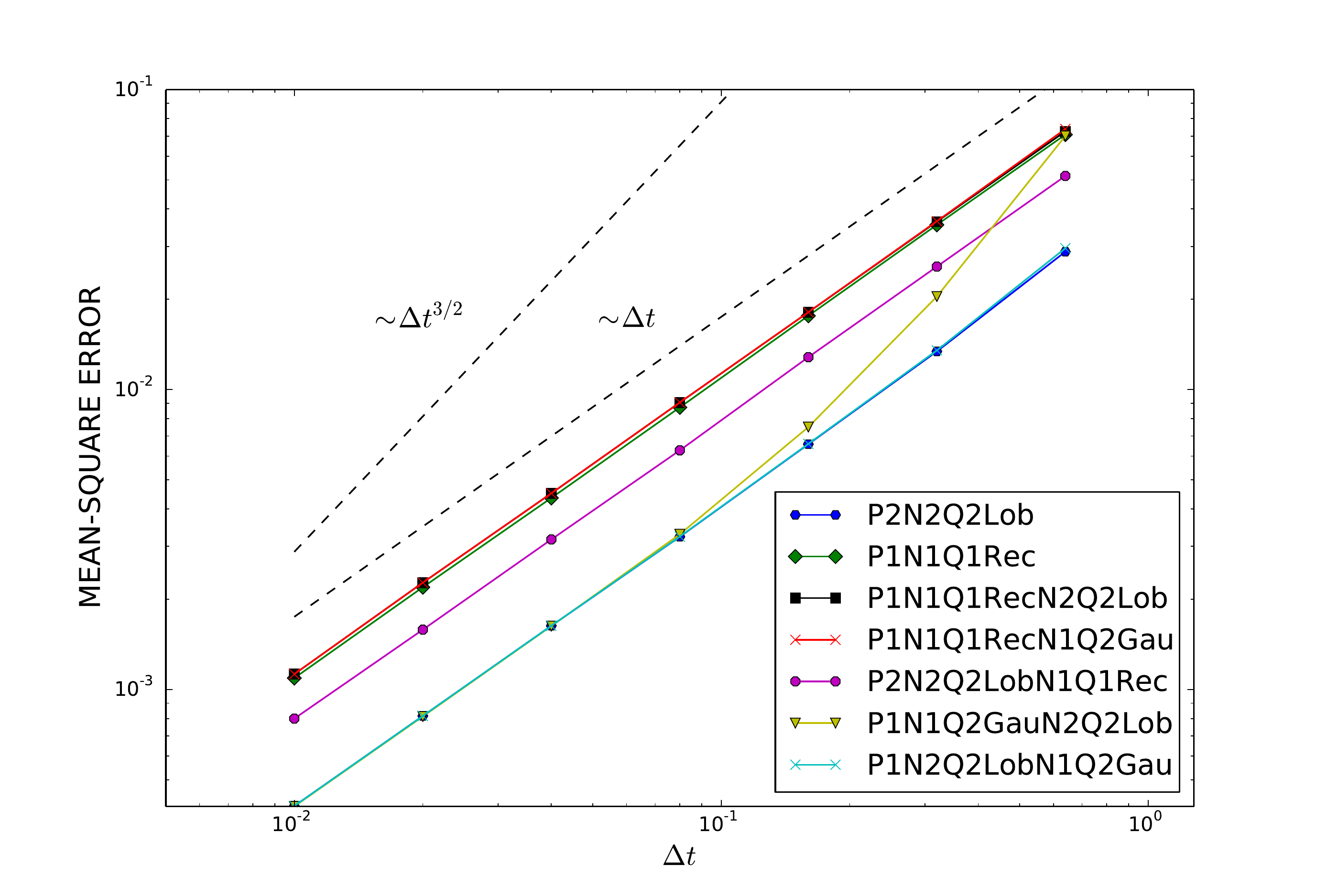}
		\caption{ The mean-square error at the time $T=3.2$ as a function of the time step $\Delta t$ for the simulations of the synchrotron oscillations of a particle in a storage ring with the initial conditions $q_0=0$, $p_0=1$ and the noise intensity $\beta=0.1$. In each case 2000 sample paths were generated. The tested integrators proved to be convergent of order 1.0 in the mean-square sense. Note that the plots for $P1N1Q1Rec$, $P1N1Q1RecN2Q2Lob$, and $P1N1Q1RecN1Q2Gau$, as well as for $P2N2Q2Lob$ and $P1N2Q2LobN1Q2Gau$, overlap very closely. }
		\label{fig: Convergence plot for the synchrotron oscillations}
\end{figure}

\subsection{Long-time energy behavior}
\label{sec:Long-time energy behavior}

\subsubsection{Kubo oscillator}
\label{sec:Kubo oscillator energy}
One can easily check that in the case of the Kubo oscillator the Hamiltonian $H(q,p)$ stays constant for almost all sample paths, i.e., $H(\bar q(t), \bar p(t))=H(q_0,p_0)$ almost surely. We used this example to test the performance of the integrators from Section~\ref{sec:General stochastic Hamiltonian}. Simulations with the initial conditions $q_0=0$, $p_0=1$, the noise intensity $\beta=0.1$, and the relatively large time step $\Delta t = 0.25$ were carried out until the time $T=1000$ (approximately 160 periods of the oscillator in the absence of noise) for a single realization of the Wiener process. For comparison, similar simulations were carried out using non-symplectic explicit methods like Milstein's scheme and the order 3/2 strong Taylor scheme (see \cite{KloedenPlatenSDE}). The numerical value of the Hamiltonian $H(q,p)$ as a function of time for each of the integrators is depicted in Figure~\ref{fig: Hamiltonian for Kubo Oscillator}. We find that the non-symplectic schemes do not preserve the Hamiltonian well, even if small time steps are used. For example, we find that Milstein's scheme does not give a satisfactory solution even with $\Delta t = 0.001$, and though the Taylor scheme with $\Delta t=0.05$ yields a result comparable to the variational integrators, the growing trend of the numerical Hamiltonian is evident. On the other hand, the variational integrators give numerical solutions for which the Hamiltonian oscillates around the true value (one can check via a direct calculation that the stochastic midpoint method \eqref{eq:Stochastic midpoint method} in this case preserves the Hamiltonian exactly; of course this does not necessarily hold in the general case).

\begin{figure}
	\centering
		\includegraphics[width=\textwidth]{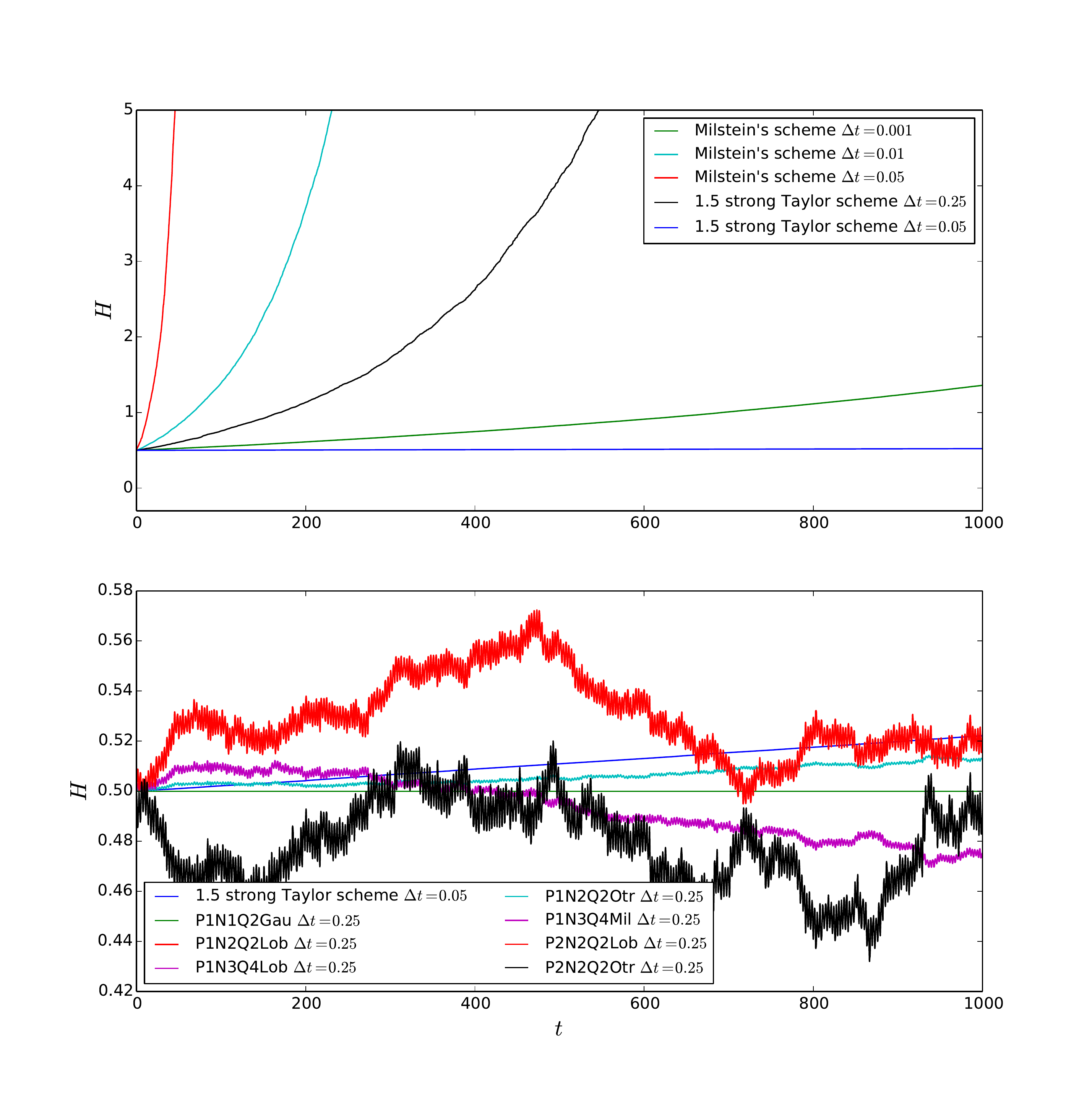}
		\caption{The numerical Hamiltonian for the simulations of the Kubo oscillator with the initial conditions $q_0=0$, $p_0=1$ and the noise intensity $\beta=0.1$. \emph{Top}: The results obtained with Milstein's scheme and the order 3/2 strong Taylor scheme. We see that the Hamiltonian tends to blow up despite using small time steps.  \emph{Bottom}: The results obtained with the integrators derived in Section~\ref{sec:General stochastic Hamiltonian}. For comparison, the solution obtained with the Taylor scheme for $\Delta t=0.05$ is also included. Note that for clarity the same color code is applied when the plots for some integrators overlap very closely.}
		\label{fig: Hamiltonian for Kubo Oscillator}
\end{figure}

\subsubsection{Anharmonic oscillator}
\label{sec:Anharmonic oscillator}
In general the Hamiltonian $H(q,p)$ does not stay constant for stochastic Hamilton equations. To determine how well our integrators perform in such cases we considered the anharmonic oscillator defined by $H(q,p)=p^2/2+\gamma q^4$ and $h(q)=\beta q$, where $\beta$ is the noise intensity and $\gamma$ is a parameter. One can calculate the expected value of the Hamiltonian analytically as

\begin{equation}
\label{eq:Expected value of the Hamiltonian for the anharmonic oscillator}
E\Big( H\big( q(t),p(t) \big) \Big) = H\big( q_0,p_0 \big) + \frac{\beta^2}{2}t,
\end{equation}

\noindent
that is, the mean value of the Hamiltonian grows linearly in time (see \cite{SeesselbergParticleStorageRings}). Simulations with the initial conditions $q_0=0$, $p_0=1$, the parameter $\gamma=0.1$, and the noise intensity $\beta=0.1$ were carried out until the time $T=784$ (approximately 100 periods of the oscillator in the absence of noise). In each case 10,000 sample paths were generated. The numerical value of the mean Hamiltonian $E(H)$ as a function of time for each of the integrators is depicted in Figure~\ref{fig: Hamiltonian for Anharmonic Oscillator}. We see that the variational integrators accurately capture the linear growth of $E(H)$, whereas the Taylor scheme fails to reproduce that behavior even when a smaller time step is used. It is worth noting that the integrators $P1N1Q1RecN2Q2Lob$ and $P1N1Q1RecN1Q2Gau$ yield a very accurate solution, while being computationally efficient, as discussed in Section~\ref{sec:Stochastic Hamiltonian independent of momentum}.

\begin{figure}
	\centering
		\includegraphics[width=0.9\textwidth]{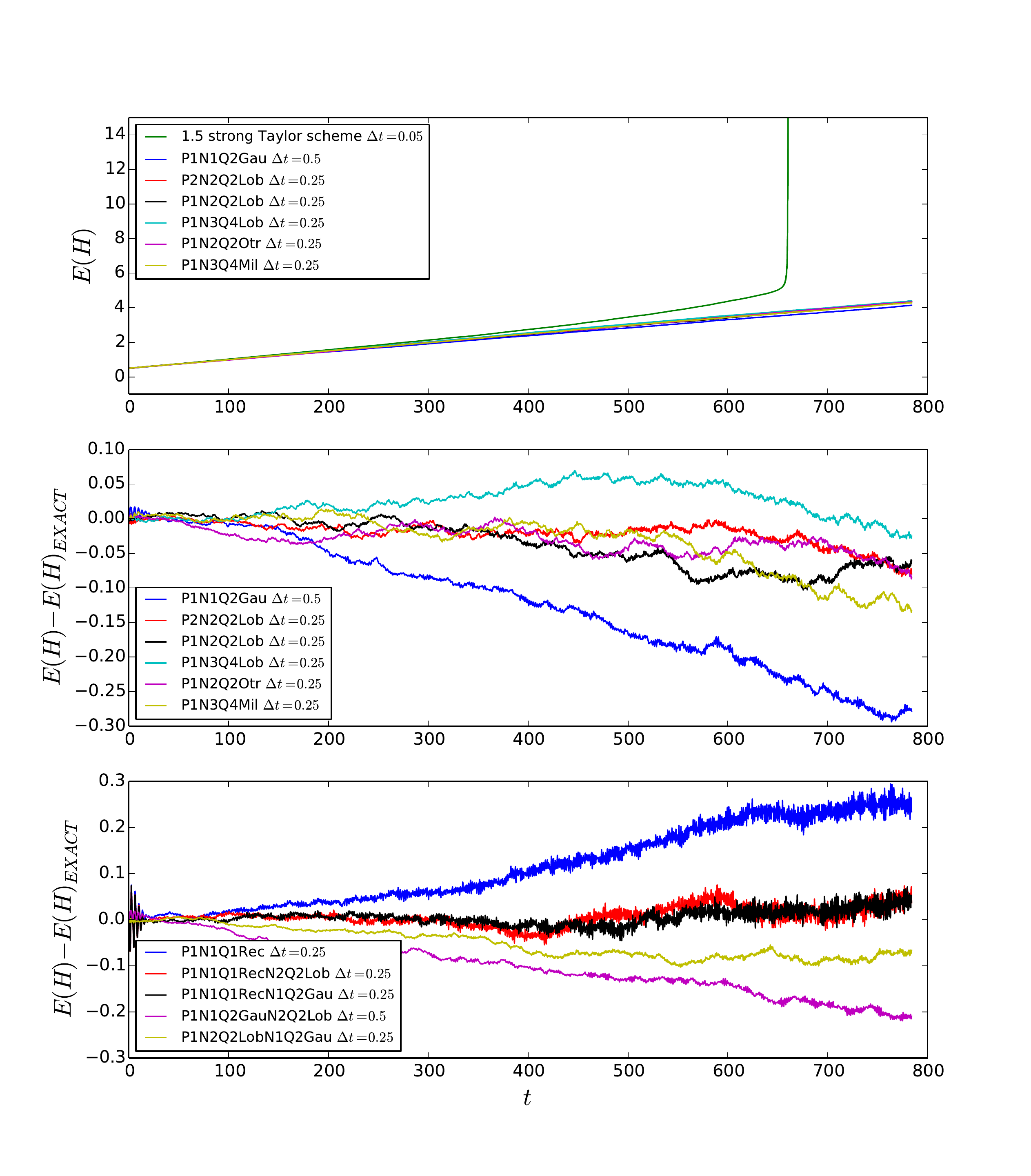}
		\caption{ \emph{Top:} The numerical value of the mean Hamiltonian $E(H)$ for the simulations of the anharmonic oscillator with the initial conditions $q_0=0$, $p_0=1$, the parameter $\gamma=0.1$, and the noise intensity $\beta=0.1$ is shown for the solutions computed with the order 3/2 strong Taylor scheme using the time step $\Delta t=0.05$ and the variational integrators derived in Section~\ref{sec:General stochastic Hamiltonian} using the time step $\Delta t=0.25$ or $\Delta t=0.5$. The variational integrators accurately capture the linear growth of $E(H)$, whereas the Taylor scheme fails to reproduce that behavior. \emph{Middle:} The difference between the numerical value of the mean Hamiltonian $E(H)$ and the exact value \eqref{eq:Expected value of the Hamiltonian for the anharmonic oscillator} is shown for the integrators derived in Section~\ref{sec:General stochastic Hamiltonian}. \emph{Bottom:} Same for the integrators derived in Section~\ref{sec:Stochastic Hamiltonian independent of momentum}. The integrators $P1N1Q1RecN2Q2Lob$ and $P1N1Q1RecN1Q2Gau$ prove to be particularly accurate, while having a low computational cost.}
		\label{fig: Hamiltonian for Anharmonic Oscillator}
\end{figure}

\paragraph{Remark.} One can verify by a direct calculation that when the $P2N2Q2Otr$ integrator (example~6 in Section~\ref{sec:General stochastic Hamiltonian}) is applied to the Kubo oscillator, then the corresponding system of equations \eqref{eq:SPRK for stochastic Hamiltonian systems} does not have a solution when $\Delta t + \beta \Delta W = 3$. To avoid numerical difficulties, one could in principle use the truncated increments \eqref{eq:Truncated Wiener increments} with, e.g., $A=(3-\Delta t)/(2\beta)$ (for $\Delta t <3$). However, given the negligible probability that $|\Delta W|>A$ for the parameters used in Section~\ref{sec:Kubo oscillator convergence} and Section~\ref{sec:Kubo oscillator energy}, we did not observe any numerical issues, even though we did not use truncated increments. In the case of all the other numerical experiments presented in Section~4, the applied algorithms either turned out to be explicit, or the corresponding nonlinear systems of equations had solutions for all values of $\Delta W$. Nonlinear equations were solved using Newton's method and the previous time step values of the position $q_k$ and momentum $p_k$ were used as initial guesses.

\section{Summary}
\label{sec:Summary}
In this paper we have presented a general framework for constructing a new class of stochastic symplectic integrators for stochastic Hamiltonian systems. We generalized the approach of Galerkin variational integrators introduced in \cite{LeokZhang}, \cite{MarsdenWestVarInt}, \cite{OberBlobaum2015} to the stochastic case, following the ideas underlying the stochastic variational integrators introduced in \cite{BouRabeeSVI}. The solution of the stochastic Hamiltonian system was approximated by a polynomial of degree $s$, and the action functional was approximated by a quadrature formula based on $r$ quadrature points. We showed that the resulting integrators are symplectic, preserve integrals of motion related to Lie group symmetries, and include stochastic symplectic Runge-Kutta methods introduced in \cite{MaDing2012}, \cite{MaDing2015}, \cite{MilsteinRepin} as a special case when $r=s$. We pointed out several new low-stage stochastic symplectic methods of mean-square order 1.0 for systems driven by a one-dimensional noise, both for the case of a general Hamiltonian function $h=h(q,p)$ and a Hamiltonian function $h=h(q)$ independent of $p$, and demonstrated their superior long-time numerical stability and energy behavior via numerical experiments. We also stated the conditions under which these integrators retain their first order of convergence when applied to systems driven by a multidimensional noise.

Our work can be extended in several ways. In Section~\ref{sec:Methods of order 3/2} we indicated how higher-order stochastic variational integrators can be constructed and showed that a type of stochastic symplectic partitioned Runge-Kutta methods of mean-square order 3/2 considered in \cite{MilsteinRepin} can be recast in that formalism. It would be interesting to derive new stochastic integrators of order 3/2 by choosing appropriate values for the parameters in \eqref{eq:Discrete Hamiltonian with dZ} or \eqref{eq:Discrete Hamiltonian with dZ for separable Hamiltonians}. It would also be interesting to apply the Galerkin approach to construct stochastic variational integrators for constrained (see \cite{BouRabeeConstrainedSVI}) and dissipative (see \cite{BouRabeeOwhadi2010}) stochastic Hamiltonian systems, and systems defined on Lie groups (see \cite{LeokShingel}). Another important problem would be stochastic variational error analysis. That is, rather than considering how closely the numerical solution follows the exact trajectory of the system, one could investigate how closely the discrete Hamiltonian matches the exact generating function. In the deterministic setting these two notions of the order of convergence are equivalent (see \cite{MarsdenWestVarInt}). It would be instructive to know if a similar result holds in the stochastic case. A further vital task would be to develop higher-order weakly convergent stochastic variational integrators. As mentioned in Section~\ref{sec:Construction of the integrator} and Section~\ref{sec:Methods of order 3/2}, higher-order methods require inclusion of higher-order multiple Stratonovich integrals, which are cumbersome to simulate in practice. In many cases, though, one is only interested in calculating the probability distribution of the solution rather than precisely approximating each sample path. In such cases weakly convergent methods are much easier to use (see \cite{KloedenPlatenSDE}, \cite{MilsteinBook}). Finally, one may extend the idea of variational integration to stochastic multisymplectic partial differential equations such as the stochastic Korteweg-de Vries, Camassa-Holm or Hunter-Saxton equations. Theoretical groundwork for such numerical schemes has been recently presented in \cite{HolmTyranowskiVirasoro}.

\section*{Acknowledgments}
We would like to thank Nawaf Bou-Rabee, Mickael Chekroun, Dan Crisan, Nader Ganaba, Melvin Leok, Juan-Pablo Ortega, Houman Owhadi, and Wei Pan for useful comments and references. This work was partially supported by the European Research Council Advanced Grant 267382 FCCA. Parts of this project were completed while the authors were visiting the Institute for Mathematical Sciences, National University of Singapore in 2016.



\end{document}